\documentclass[sn-mathphys-num]{sn-jnl}
\usepackage{graphicx}
\usepackage{multirow}
\usepackage{amsmath,amssymb,amsfonts}
\usepackage{amsthm}
\usepackage{mathrsfs}
\usepackage[title]{appendix}
\usepackage{xcolor}
\usepackage{textcomp}
\usepackage{manyfoot}
\usepackage{booktabs}
\usepackage{algorithm}
\usepackage{algorithmicx}
\usepackage{algpseudocode}
\usepackage{listings}
\newcommand{\doiurl}[1]{\href{https://doi.org/#1}{\nolinkurl{#1}}}
\newcommand{\weburl}[1]{\href{#1}{\nolinkurl{#1}}}

\theoremstyle{thmstyleone}
\newtheorem{theorem}{Theorem}
\newtheorem{proposition}[theorem]{Proposition}
\theoremstyle{thmstyletwo}
\newtheorem{example}{Example}
\newtheorem{remark}{Remark}
\newtheorem{corollary}{Corollary}
\newtheorem{lemma}{Lemma}
\theoremstyle{thmstylethree}
\newtheorem{definition}{Definition}

\begin{document}

\title[Mutually orthogonal anti-Latin squares]{Mutually orthogonal anti-Latin squares}

\author[1]{\fnm{Eishiro} \sur{Aoyama}}
\equalcont{Alphabetical order}
\author[1]{\fnm{So} \sur{Hasegawa}}
\author*[2,3]{\fnm{Masahito} \sur{Hayashi}}\email{hmasahito@cuhk.edu.cn; masahito@math.nagoya-u.ac.jp}
\author[1]{\fnm{Tomoki} \sur{Sagara}}

\affil[1]{\orgname{Tokai Senior High School}, \orgaddress{\street{Higashi-ku}, \city{Nagoya}, \postcode{461-0003}, \state{Aichi}, \country{Japan}}}

\affil[2]{\orgdiv{School of Data Science}, \orgname{The Chinese University of Hong Kong, Shenzhen}, \orgaddress{\street{Longgang District}, \city{Shenzhen}, \postcode{518172}, \state{Guangdong}, \country{China}}}

\affil[3]{\orgdiv{Graduate School of Mathematics}, \orgname{Nagoya University}, \orgaddress{\street{Chikusa-ku}, \city{Nagoya}, \postcode{464-8602}, \state{Aichi}, \country{Japan}}}

\abstract{
Anti-Latin squares were introduced in connection with non-linear secure network coding, and the extremal problem for large mutually orthogonal families is motivated by that setting. We study the maximum size $N_A(d)$ of a family of mutually orthogonal anti-Latin squares of order $d$. We prove that $N_L(d)+1\le N_A(d)\le N_L(d)+2$ for every $d\ge 3$, where $N_L(d)$ denotes the classical maximum size of a family of mutually orthogonal Latin squares of order $d$, and we show that in fact $N_A(3)=N_L(3)+1$ whereas $N_A(d)=N_L(d)+2$ for every $d\ge 4$. The upper bound is obtained by passing through balanced matrices, while the lower bound is given by a deterministic permutation argument. For all $d\ge 8$, and also for the exceptional order $d=6$, the upper bound is shown to be attainable by a general probabilistic construction. On the structural side, we show that a saturated family of size $d+1$ induces an affine plane of order $d$, and that the saturated case is characterized by the existence of an anti-coordinate grid decomposition; after transporting this condition to the fixed cell set $[d]^2$, it becomes a direction-completeness condition on the corresponding row-blocks and column-blocks. The remaining small orders are treated separately: $d=3$ is handled by direct analysis and classification of orthogonal triples, $d=4$ by an explicit saturated construction and an analysis of its finite-geometric structure, and $d=5$ and $d=7$ by explicit saturated examples arising from the random-grid framework. Thus $N_A(d)$ is determined in terms of $N_L(d)$ for every $d\ge3$, and its
numerical value is obtained explicitly for every $3\le d\le9$.
}


\maketitle

\section{Introduction}\label{S1}

Latin squares and their mutually orthogonal families are classical objects in combinatorics and design theory \cite{Euler,Fisher,Preece}. A central extremal quantity is $N_L(d)$, the maximum size of a family of mutually orthogonal Latin squares of order $d$. When $d$ is a prime power, it is known that $N_L(d)=d-1$ \cite{MacNeish,Mann}. More generally, if $d=\prod_{j=1}^k p_j^{n_j}$, then the classical MacNeish lower bound asserts that $N_L(d)\ge \min_j p_j^{n_j}-1$ \cite{MacNeish,Mann}. The exact determination of $N_L(d)$ is in general difficult and remains a longstanding problem in combinatorics \cite{Bose,Bose2,Roberts,Kouvela,Dey,Benad,Boyadzhiyska,Jager,Donovan}.

In this paper, we study the corresponding extremal problem for anti-Latin squares. Anti-Latin squares were introduced in connection with non-linear secure network coding \cite{H1,H2}. In that setting, orthogonality is needed for correct transmission, while a biased arrangement is useful for concealing information from an eavesdropper. Anti-Latin squares capture this non-uniformity, and mutually orthogonal anti-Latin squares therefore arise naturally as combinatorial objects in secure network coding \cite{H1,H2}.

A $d\times d$ matrix on a set of $d$ symbols is called an anti-Latin square of order $d$ if at least one symbol repeats in every row and in every column, and each symbol appears exactly $d$ times in the whole matrix. For example,
\[
A_1 :=
\left(
\begin{array}{ccc}
1&1&2\\
0&2&2\\
0&1&0
\end{array}
\right)
\]
is an anti-Latin square of order $3$. Two anti-Latin squares $(a_{i,j})$ and $(b_{i,j})$ are called orthogonal if the map $(i,j)\mapsto (a_{i,j},b_{i,j})$ is one-to-one. For instance, the two anti-Latin squares
\[
A_2 :=
\left(
\begin{array}{ccc}
1&0&0\\
1&2&1\\
2&2&0
\end{array}
\right),
\qquad
A_3 :=
\left(
\begin{array}{ccc}
2&1&2\\
1&1&0\\
2&0&0
\end{array}
\right)
\]
are orthogonal. Orthogonal pairs of anti-Latin squares already play a key role in the construction of non-linear secure network codes over the one-hop relay network, and examples exist for every integer $d\ge 3$ \cite{H1}.

Let ${\cal A}(d)$ denote the set of anti-Latin squares of order $d$, and let $N_A(d)$ denote the maximum size of a mutually orthogonal family in ${\cal A}(d)$. The present paper addresses the extremal problem of determining $N_A(d)$ and clarifies its relation to the classical Latin-square quantity $N_L(d)$ \cite{H2}.

Our first result shows that the anti-Latin extremal problem is always very close to the classical one: for every $d\ge 3$, we have $N_L(d)+1\le N_A(d)\le N_L(d)+2$. The main extremal result of the paper resolves this ambiguity completely, as stated in the following theorem.
\begin{theorem}\label{thm:main-exact-formula}
We have
\[
N_A(3)=N_L(3)+1,\qquad
N_A(d)=N_L(d)+2 \quad (d\ge 4).
\]
\end{theorem}

Thus, apart from the single exceptional order $d=3$, the maximum number of mutually orthogonal anti-Latin squares is always exactly two larger than the corresponding Latin-square extremal number.

For clarity, the contributions of the paper may be viewed on three levels.
First, we establish the universal comparison between $N_A(d)$ and $N_L(d)$
and derive the exact formula in Theorem~\ref{thm:main-exact-formula}.
Second, we identify the geometric structure underlying saturated families,
that is, families of size $d+1$.
Third, we complete the low-order analysis: we classify the extremal triples
for $d=3$, analyze the additional finite-geometric structure forced in the
saturated case $d=4$, and give explicit saturated constructions for $d=5$
and $d=7$.

The geometric ingredients used in the saturated case belong to several
classical correspondences in finite design theory. Complete nets, affine
planes, orthogonal arrays of strength two, and complete families of mutually
orthogonal Latin squares are different representations of closely related
incidence structures
\cite{Bruck1951,Dembowski,Stinson,BoseBush,AbelColbournDinitzMOLS}.
Our contribution is not a new version of this classical correspondence.
Rather, we show that the anti-Latin condition selects an additional pair of
orthogonal resolutions of the point set such that none of their blocks is
transversal to any parallel class. Equivalently, a saturated family is
characterized by an affine plane equipped with what we call an
anti-coordinate grid decomposition.

In the Desarguesian model $AG(2,q)$, this additional condition can be
expressed in terms of directions determined by point sets. The classical
direction problem studies the number and structure of directions determined
by a single affine point set
\cite{LovaszSchrijver,SzonyiDirections,SomlaiRedei,KissSomlaiSpecialDirections}.
The present problem instead requires two orthogonal partitions of
$AG(2,q)$ into $q$-point sets, with every block in both partitions
determining all $q+1$ directions. Thus the relevant object is a globally
compatible system of direction-complete sets, rather than a single point set
with a prescribed set of directions.

The proof combines general arguments with separate treatments of the small
orders. We first obtain the upper bound by passing through balanced matrices
and the lower bound by a deterministic cell permutation. A probabilistic
cell-permutation argument attains the upper bound for every $d\ge 8$ and also
for $d=6$. The remaining cases are handled individually: $d=3$ by direct
analysis and classification, $d=4$ by an explicit saturated construction and
a finite-geometric analysis, and $d=5$ and $d=7$ by explicit constructions
arising from the random-grid framework. Consequently, the numerical value of
$N_A(d)$ is determined explicitly for every $3\le d\le 9$.

The paper is organized as follows.
Section~\ref{sec:deterministic} establishes the general extremal bounds by
introducing balanced matrices, proving $N_A(d)\le N_L(d)+2$, and obtaining the
deterministic lower bound $N_A(d)\ge N_L(d)+1$.
Section~\ref{Sec3} gives the probabilistic cell-permutation argument for
$d\ge 8$ and $d=6$, and records the random-grid sampling procedure used for
the explicit prime-power constructions.
Section~\ref{sec:geom} develops the geometric reformulation, passing from
orthogonal balanced matrices to nets and characterizing saturated families by
affine planes with anti-coordinate grid decompositions; it then gives the
corresponding direction formulation on the fixed cell set $[d]^2$.
Section~\ref{sec:small-ex} treats the remaining small orders, including the
classification for $d=3$ and the geometric analysis for $d=4$, and completes
the proof of the main theorem. The appendices provide the explicit examples
for $d=5$ and $d=7$ together with reproducibility information.
\section{Deterministic approach}\label{sec:deterministic}
\subsection{Upper bound for $N_A(d)$}
\label{sec:upper}
Throughout this paper, we employ the symbol set $[d]:=\{0,1,\dots,d-1\}$. 
We begin this section by focusing on
a mutually orthogonal family of balanced matrices.

\begin{definition}[Balanced matrix]\label{def:balanced-matrix-general}
Let $d\ge 2$.
A $d\times d$ matrix over a symbol set of size $d$ is called
\emph{balanced} if each symbol appears exactly $d$ times in the matrix.
\end{definition}

Let $N_B(d)$ denote the maximum size 
of a mutually orthogonal family of balanced matrices.
Given a mutually orthogonal family of Latin squares,
$B^{1}, \ldots, B^{K'}$, 
we define the following matrices $R$ and $C$ as
\begin{align}
R_{k,j}=k,\quad
C_{k,j}=j
\end{align}
for $k,j \in[d]$.
The definition of Latin squares implies that 
$B^{1}, \ldots, B^{K'},R,C$ are mutually orthogonal.
Hence, we have $N_B(d)\ge N_L(d)+2$.
Conversely, let
$B^{1},\ldots,B^{K}$ be a mutually orthogonal family of balanced matrices.
Orthogonality of $B^{1}$ and $B^{2}$ implies that the map
\[
 x\longmapsto \bigl(B^{1}_x,B^{2}_x\bigr)
 \qquad (x\in[d]^2)
\]
is a bijection from the cell set onto $[d]^2$.  Let
$\sigma:[d]^2\to[d]^2$ be its inverse, with the two coordinates ordered so
that
\[
 B^{1}_{\sigma(k,j)}=k,
 \qquad
 B^{2}_{\sigma(k,j)}=j.
\]
For a matrix $B$, write $\sigma(B)_{k,j}:=B_{\sigma(k,j)}$.  Then
$\sigma(B^{1})=R$ and $\sigma(B^{2})=C$.  For every $l\ge3$, orthogonality of
$B^l$ with $B^1$ and $B^2$ implies that each symbol occurs exactly once in
every row and every column of $\sigma(B^l)$; hence
$\sigma(B^{3}),\ldots,\sigma(B^{K})$ form a mutually orthogonal family of
Latin squares.  Therefore $K-2\le N_L(d)$, which gives
$N_B(d)\le N_L(d)+2$.  Thus, we have
\begin{align}
N_B(d)= N_L(d)+2.
\end{align}
Since the definition of an anti-Latin square contains the balanced condition, 
the relation $N_A(d) \le N_B(d) $ holds. 
Therefore, we have the following theorem. 
\begin{theorem}\label{thm:NA-upper-via-MOLS}
The relation 
\begin{align}
N_A(d)\le N_L(d)+2 \label{BV1}
\end{align}
holds.
\end{theorem}

\subsection{Lower bound for $N_A(d)$}\label{S2}
We next prove the deterministic lower bound.
\begin{theorem}\label{thm:3}
For every $d\ge3$,
\begin{align}
N_A(d)\ge N_L(d)+1. \label{BV3}
\end{align}
\end{theorem}
\begin{proof}
Let $B^1,\ldots,B^{K'}$ be a mutually orthogonal family of Latin squares of
maximum size $K'=N_L(d)$, and let $R$ and $C$ be the row- and column-coordinate
matrices
\[
 R_{k,j}=k,
 \qquad
 C_{k,j}=j
 \qquad(k,j\in[d]).
\]
Then $B^1,\ldots,B^{K'},R,C$ are mutually orthogonal.  It is enough to
construct one permutation of the cell set such that
$B^2,\ldots,B^{K'},R,C$ become anti-Latin, because a common cell permutation
preserves balance and pairwise orthogonality.

Fix the symbol $0$ in $B^1$ and set
\[
 \Lambda:=\{(k,j)\in[d]^2:B^1_{k,j}=0\}.
\]
Since $B^1$ is Latin, $\Lambda$ meets every row and every column in exactly one
cell.  Thus there is a permutation $j:[d]\to[d]$ such that
\[
 \Lambda=\{(k,j(k)):k\in[d]\}.
\]
All additions in the row index below are taken modulo $d$.  Define a
permutation $\sigma$ of $[d]^2$ by cyclically shifting the cells of $\Lambda$,
\[
 \sigma(k,j(k))=(k+1,j(k+1)),
\]
and fixing every cell outside $\Lambda$.  For any array $A$, write
\[
 \sigma(A)_{k,j}:=A_{\sigma(k,j)}.
\]

We first consider a Latin square $B^l$ with $l\ge2$.  Orthogonality of $B^1$
and $B^l$ implies that the values
\[
 B^l_{k,j(k)},\qquad k\in[d],
\]
are all distinct, because they occur in the cells where $B^1=0$.  Hence
\[
 B^l_{k+1,j(k+1)}\ne B^l_{k,j(k)}.
\]
In row $k$, the only changed cell is $(k,j(k))$.  Its new value is
$B^l_{k+1,j(k+1)}$, which differs from the removed value.  Since row $k$ of
$B^l$ originally contains every symbol exactly once, the new value already
occurs in one of the other, unchanged cells of that row.  Therefore row $k$ of
$\sigma(B^l)$ contains a repeated symbol.

Now fix a column $c$.  Since $j$ is a permutation, there is a unique $k$ with
$c=j(k)$, and $(k,c)$ is the only changed cell in column $c$.  Its new value
again differs from the removed value.  Since column $c$ of $B^l$ originally
contains every symbol exactly once, that new value occurs in another cell of
the column, and this other cell is unchanged.  Hence every column of
$\sigma(B^l)$ also contains a repeated symbol.  Thus $\sigma(B^l)$ is
anti-Latin for every $l\ge2$.

It remains to check the coordinate matrices.  In row $k$ of $\sigma(C)$, the
entry at column $j(k)$ changes from $j(k)$ to $j(k+1)$.  Since
$j(k+1)\ne j(k)$, the same value $j(k+1)$ also occurs at the unchanged cell
$(k,j(k+1))$.  Thus every row of $\sigma(C)$ has a repetition.  In each column
$c=j(k)$, all cells except $(k,c)$ still contain the value $c$; since $d\ge3$,
this gives a repetition in every column.  Hence $\sigma(C)$ is anti-Latin.

Similarly, in row $k$ of $\sigma(R)$, all cells outside $(k,j(k))$ retain the
value $k$, so every row has a repetition.  In column $j(k)$, the changed cell
has value $k+1$, and the unchanged cell in row $k+1$ of the same column also
has value $k+1$ because $j(k+1)\ne j(k)$.  Thus every column has a repetition,
and $\sigma(R)$ is anti-Latin.

Consequently,
\[
 \sigma(B^2),\ldots,\sigma(B^{K'}),\sigma(R),\sigma(C)
\]
form a mutually orthogonal family of $K'+1=N_L(d)+1$ anti-Latin squares.
\end{proof}
\section{Probabilistic approach}\label{Sec3}
\subsection{The attainability of \eqref{BV1}}
Let $B^1,\ldots,B^K$ be a mutually orthogonal family of balanced matrices on
the cell set $[d]^2$.  A common permutation of the cells preserves balance and
pairwise orthogonality.  We show that, for $d\ge8$ and also for $d=6$, a
uniformly random cell permutation has positive probability of making every
matrix anti-Latin.

\begin{theorem}\label{thm:NA-equals-NL-plus-2-for-all-d-ge-8}
When $d\ge8$ or $d=6$, we have
\[
 N_A(d)=N_L(d)+2.
\]
\end{theorem}

Choose a permutation $\sigma$ of $[d]^2$ uniformly at random and define
$\sigma(B)_{k,j}:=B_{\sigma(k,j)}$.  For $k,j\in[d]$ and
$l\in\{1,\ldots,K\}$, let
\begin{description}
\item[$E^{\rm row}_{k,l}$] be the bad event that row $k$ of $\sigma(B^l)$ has
no repeated symbol;
\item[$E^{\rm col}_{j,l}$] be the bad event that column $j$ of $\sigma(B^l)$
has no repeated symbol.
\end{description}

\begin{lemma}\label{lem:bad-row-column-probability}
For every $k,j\in[d]$ and $l\in\{1,\ldots,K\}$,
\begin{align}
 \Pr(E^{\rm row}_{k,l})
 &=\frac{d^d}{\binom{d^2}{d}}, \label{NM1}\\
 \Pr(E^{\rm col}_{j,l})
 &=\frac{d^d}{\binom{d^2}{d}}. \label{NM2}
\end{align}
\end{lemma}
\begin{proof}
We prove the row formula; the column formula is identical.  The image under
$\sigma$ of the $d$ cells in a fixed row is a uniformly distributed
$d$-subset of the $d^2$ cells.  Hence there are $\binom{d^2}{d}$ equally likely
image sets.

Because $B^l$ is balanced, each of its $d$ symbol classes contains exactly
$d$ cells.  The selected $d$-subset has no repeated symbol exactly when it
contains one cell from each symbol class.  There are $d$ choices from each of
the $d$ classes, hence $d^d$ such subsets.  This proves \eqref{NM1}.
\end{proof}

\begin{lemma}\label{lem:one-balanced-matrix-random-anti-latin}
For every $l\in\{1,\ldots,K\}$,
\begin{align}
 \Pr\bigl[\sigma(B^l)\text{ is anti-Latin}\bigr]
 \ge 1-2d\frac{d^d}{\binom{d^2}{d}}. \label{NM3}
\end{align}
\end{lemma}
\begin{proof}
If none of the $2d$ bad events
\[
 E^{\rm row}_{k,l}\quad(k\in[d]),
 \qquad
 E^{\rm col}_{j,l}\quad(j\in[d])
\]
occurs, then every row and every column of $\sigma(B^l)$ contains a repeated
symbol.  Since a cell permutation preserves balance, $\sigma(B^l)$ is then
anti-Latin.  The union bound and
Lemma~\ref{lem:bad-row-column-probability} give
\[
 1-\Pr\bigl[\sigma(B^l)\text{ is anti-Latin}\bigr]
 \le
 \sum_{k\in[d]}\Pr(E^{\rm row}_{k,l})
 +\sum_{j\in[d]}\Pr(E^{\rm col}_{j,l})
 =2d\frac{d^d}{\binom{d^2}{d}}.
\]
\end{proof}

Applying the union bound once more over the $K$ matrices gives the following.
\begin{lemma}\label{LL1}
Under a uniformly random cell permutation $\sigma$,
\begin{align}
 \Pr\bigl[\sigma(B^1),\ldots,\sigma(B^K)
 \text{ are all anti-Latin}\bigr]
 \ge 1-2Kd\frac{d^d}{\binom{d^2}{d}}. \label{NM5}
\end{align}
\end{lemma}

The next estimate shows that the right-hand side of \eqref{NM5} is positive
when $K=N_B(d)$ in the required orders.
\begin{lemma}[Numerical estimate for $d\ge8$ and $d=6$]
\label{lem:general-numerical-estimate}
When $d\ge8$ or $d=6$,
\[
 2N_B(d)d\frac{d^d}{\binom{d^2}{d}}<1. \label{NM6}
\]
\end{lemma}
\begin{proof}
Since $N_B(d)\le d+1$, it is sufficient for $d\ge8$ to show
\[
 2d(d+1)\frac{d^d}{\binom{d^2}{d}}<1.
\]
Using
\[
 \binom{d^2}{d}
 =\frac{d^{2d}}{d!}
  \prod_{t=0}^{d-1}\left(1-\frac{t}{d^2}\right),
\]
we obtain
\[
 \frac{d^d}{\binom{d^2}{d}}
 =\frac{d!}{d^d}
  \frac{1}{\displaystyle\prod_{t=0}^{d-1}
  \left(1-\frac{t}{d^2}\right)}.
\]
The elementary inequality
$\prod_t(1-a_t)\ge1-\sum_ta_t$ for $a_t\in[0,1]$ gives
\[
 \prod_{t=0}^{d-1}\left(1-\frac{t}{d^2}\right)
 \ge1-\sum_{t=0}^{d-1}\frac{t}{d^2}
 =\frac{d+1}{2d}.
\]
Therefore
\[
 2d(d+1)\frac{d^d}{\binom{d^2}{d}}
 \le4d^2\frac{d!}{d^d}=:c_d.
\]
Moreover,
\[
 \frac{c_{d+1}}{c_d}
 =\left(1+\frac1d\right)^{2-d}<1
 \qquad(d\ge3),
\]
so $(c_d)$ is decreasing, and
\[
 c_8=4\cdot8^2\frac{8!}{8^8}=\frac{315}{512}<1.
\]
Thus the desired inequality holds for every $d\ge8$.

For $d=6$, the classical equality $N_L(6)=1$ gives $N_B(6)=3$
\cite{Tarry1901,Stinson1984}.  Directly,
\[
 2\cdot3\cdot6\frac{6^6}{\binom{36}{6}}
 =\frac{34992}{40579}<1.
\]
\end{proof}

\begin{proof}[Proof of Theorem~\ref{thm:NA-equals-NL-plus-2-for-all-d-ge-8}]
Take a mutually orthogonal family of $N_B(d)$ balanced matrices.  By
Lemma~\ref{lem:general-numerical-estimate}, the lower bound in \eqref{NM5} is
positive.  Hence some cell permutation makes all $N_B(d)$ matrices
anti-Latin.  Therefore
\[
 N_A(d)\ge N_B(d)=N_L(d)+2.
\]
The reverse inequality is Theorem~\ref{thm:NA-upper-via-MOLS}.
\end{proof}
\subsection{Concrete algorithm: random-grid sampling}\label{sec:random-grid}
Let $q$ be a prime power.  On $\mathbb F_q^2$, define
\[
 \pi_m(x,y):=y-mx\quad(m\in\mathbb F_q),
 \qquad
 \pi_\infty(x,y):=x.
\]
The fibres of the $q+1$ maps $\pi_m$ are the parallel classes of the standard
affine plane.  Equivalently, after labeling their fibres, the arrays
$L^m_{x,y}:=\pi_m(x,y)$ form a mutually orthogonal family of $q+1$ balanced
matrices.

A trial of the random-grid procedure samples a uniformly random bijection
\[
 \tau:[q]^2\longrightarrow\mathbb F_q^2.
\]
For every external row and column and every direction $m$, it tests whether
the corresponding $q$-point image has two points in one fibre of
$\pi_m$.  If all tests succeed, the line-membership arrays reconstructed from
$\tau$ are anti-Latin and mutually orthogonal.

\begin{algorithm}[t]
\caption{Random-grid sampling on $\mathbb F_q^2$}
\label{alg:random-grid}
\begin{algorithmic}[1]
\Require A prime power $q$ and a positive trial limit $T$
\Ensure Either a family of $q+1$ mutually orthogonal anti-Latin squares of
order $q$, or failure
\For{$t=1,\ldots,T$}
 \State Sample a uniformly random bijection
 $\tau:[q]^2\to\mathbb F_q^2$
 \State Set \textsc{success}$\gets$ true
 \For{each $m\in\mathbb F_q\cup\{\infty\}$}
  \For{each $i,j\in[q]$}
   \State Test whether the multiset
   $\{\pi_m(\tau(i,j')):j'\in[q]\}$ has a repeated value
   \State Test whether the multiset
   $\{\pi_m(\tau(i',j)):i'\in[q]\}$ has a repeated value
   \If{either test fails}
    \State Set \textsc{success}$\gets$ false
    \State Skip the remaining tests and proceed to the next trial
   \EndIf
  \EndFor
 \EndFor
 \If{\textsc{success}}
  \State Define $A^{(m)}_{i,j}:=\pi_m(\tau(i,j))$ for all
  $m\in\mathbb F_q\cup\{\infty\}$ and $i,j\in[q]$
  \State \Return $\{A^{(m)}\}_{m\in\mathbb F_q\cup\{\infty\}}$
 \EndIf
\EndFor
\State \Return failure
\end{algorithmic}
\end{algorithm}

The algorithm is a finite search procedure.  The probabilistic argument above
is not needed to verify an output: once the arrays are displayed, their
balancedness, anti-Latin row/column condition, and pairwise orthogonality can
be checked directly.
\section{Geometric reformulation}\label{sec:geom}
In this section, we extract the geometric structure underlying families of
pairwise orthogonal balanced matrices, and hence in particular families of
mutually orthogonal anti-Latin squares.

The passage from pairwise orthogonal balanced matrices to a net is the
standard symbol-class construction, and the fact that a complete
$(d,d+1)$-net is an affine plane is classical
\cite{Bruck1951,Dembowski,Stinson,Evans2018}.  Likewise, the connections among
complete nets, affine planes, orthogonal arrays of strength two, and complete
families of mutually orthogonal Latin squares are well established
\cite{BoseBush,AbelColbournDinitzMOLS,Stinson}.

The new point here is the treatment of the external row and column
partitions.  In the classical incidence structure they are auxiliary
resolutions of the point set, whereas the anti-Latin condition requires every
block of both resolutions to fail to be a transversal to every parallel
class.  Theorem~\ref{TH4} identifies this additional incidence condition
exactly.  It thereby separates the classical affine-plane structure carried
by orthogonality and balance from the extra simultaneous resolution condition
carried by the anti-Latin requirement.

Accordingly, the purpose of the present section is to pass from orthogonal
balanced matrices to $(d,K)$-nets, identify the saturated case with the
affine-plane case, and formulate the additional anti-Latin condition as an
anti-coordinate grid condition.  We then compare the resulting condition with
the classical single-set direction problem and give its finite-field form.
\subsection{From orthogonal matrices to a $(d,K)$-net}
\label{subsec:general-prob-net}

We begin by recording the abstract net structure carried by a family of
pairwise orthogonal balanced matrices, in the language of finite nets
from design theory \cite{Bruck1951,Evans2018}.

\begin{definition}[$(d,K)$-net]\label{def:dK-net-general}
Let $P$ be a finite set with $|P|=d^2$.
A family of partitions
\(
\Pi_1,\Pi_2,\dots,\Pi_K
\)
of $P$ is called a \emph{$(d,K)$-net} if the following conditions hold:
\begin{enumerate}
\item for each $t\in\{1,\dots,K\}$, the partition $\Pi_t$ consists of exactly
$d$ blocks, each of size $d$;
\item for any $t\neq s$, every block of $\Pi_t$ meets every block of $\Pi_s$
in exactly one point.
\end{enumerate}
\end{definition}

The relevance of these notions is that a family of balanced pairwise orthogonal
matrices determines precisely such a net structure.

\begin{lemma}[Balanced orthogonal matrices yield a $(d,K)$-net]
\label{lem:balanced-orthogonal-net}
Let $d\ge 2$, and let
\(
A^{(1)},A^{(2)},\dots,A^{(K)}
\)
be pairwise orthogonal balanced $d\times d$ matrices over $[d]$.
Set
\(
P:=[d]^2,
\)
and for each $t\in\{1,\dots,K\}$ and $a\in[d]$, define
\(
L_a^{(t)}:=\{(u,v)\in P \mid A^{(t)}_{u,v}=a\}.
\)
Then the families
\(
\Pi_t:=\{L_a^{(t)}\}_{a\in[d]}\) with $t=1,\dots,K$
form a $(d,K)$-net on $P$.
\end{lemma}

\begin{proof}
Fix $t$.
Since $A^{(t)}$ is balanced, each set $L_a^{(t)}$ has size $d$.
Since every entry of $A^{(t)}$ has a unique symbol, the sets
$\{L_a^{(t)}\}_{a\in[d]}$ are pairwise disjoint and cover
$P=[d]^2$.
Hence $\Pi_t$ is a partition of $P$ into $d$ blocks of size $d$.

Now let $t\neq s$ and $a,b\in[d]$.
Since $A^{(t)}$ and $A^{(s)}$ are orthogonal, the ordered pair $(a,b)$ appears
exactly once among the $d^2$ positions.
Therefore, we have
\(
|L_a^{(t)}\cap L_b^{(s)}|=1.
\)
Thus $\Pi_1,\dots,\Pi_K$ form a $(d,K)$-net on $P$.
\end{proof}

\subsection{The saturated case: affine planes and anti-coordinate grids}\label{subsec:saturated-geom}

\begin{definition}[Affine plane of order $d$ {\cite{Stinson}}]
An affine plane of order $d$ is a pair
$\Pi=(P,\mathcal L)$,
where $P$ is a finite set of $d^2$ points and $\mathcal L$ is a family of subsets of $P$,
called lines, such that:
\begin{enumerate}
\item each line contains exactly $d$ points;
\item any two distinct points lie on a unique common line;
\item for every line $\ell\in\mathcal L$ and every point $x\in P\setminus \ell$,
there exists a unique line through $x$ disjoint from $\ell$.
\end{enumerate}
Two lines are called \emph{parallel} if they are disjoint.
The set of all lines is then partitioned into $d+1$ \emph{parallel classes},
each consisting of $d$ pairwise disjoint lines.
\end{definition}

By Lemma~\ref{lem:balanced-orthogonal-net}, any mutually orthogonal family of anti-Latin squares determines a $(d,K)$-net, since anti-Latin squares are balanced by definition.
In the saturated case $K=d+1$, this incidence structure is exactly an affine plane of order $d$.

\begin{corollary}\label{CO1}
If there exist $d+1$ mutually orthogonal anti-Latin squares of order $d$, then there exists an affine plane of order $d$.
\end{corollary}

\begin{proof}
A family of $d+1$ mutually orthogonal anti-Latin squares determines a $(d,d+1)$-net.
It is classical that a $(d,d+1)$-net is equivalent to an affine plane of order $d$ \cite{Bruck1951,Evans2018}.
\end{proof}

From this point on, we work on the point set of an affine plane in an abstract grid-decomposition language.
Thus the symbols $P$, $R_i$, and $C_j$ refer to an abstract point set and two orthogonal partitions of that point set, rather than to the literal rows and columns of the coordinate grid 
$[d]^2$.

Let $Q$ be any set of cardinality $d^2$.  We say that a pair of families
\[
 \mathcal R=\{R_i\}_{i\in[d]},
 \qquad
 \mathcal C=\{C_j\}_{j\in[d]}
\]
forms a \emph{grid decomposition} of $Q$ if each $R_i$ and each $C_j$ has
cardinality $d$, both $\mathcal R$ and $\mathcal C$ are partitions of $Q$,
and
\[
 |R_i\cap C_j|=1
 \qquad(i,j\in[d]).
\]
The sets $R_i$ and $C_j$ are called the row-blocks and column-blocks,
respectively.  Their intersections identify $Q$ with a $d\times d$ grid.
We apply this definition both to the point set $P$ of an affine plane and,
after transport by a bijection, to the fixed cell set $[d]^2$.

\begin{definition}
Let $\Gamma\subset\mathcal L$ be a parallel class of $\Pi$.
A subset $X\subset P$ with $|X|=d$ is called a \emph{$\Gamma$-transversal} if
the relation
$|X\cap \ell|=1$ holds for every line $\ell\in\Gamma$.
A grid decomposition $(\mathcal R,\mathcal C)$ is called an \emph{anti-coordinate grid decomposition} if no member of $\mathcal R\cup\mathcal C$ is a $\Gamma$-transversal for any parallel class $\Gamma$ of $\Pi$.
\end{definition}
There is a small but important distinction between the square-side and the
geometry-side formulations.
On the square side, a family of mutually orthogonal anti-Latin squares is
written on the fixed coordinate grid $[d]^2$, so the relevant row-blocks and
column-blocks are the standard ones,
($\{\{i\}\times [d]\}_{i\in[d]}$, $\{[d]\times \{j\}\}_{j\in[d]}$).
On the geometry side, however, we work with an abstract affine plane
$\Pi=(P,\mathcal L)$ together with an arbitrary grid decomposition
($\mathcal R=\{R_i\}_{i\in[d]}$, $\mathcal C=\{C_j\}_{j\in[d]}$)
of its point set $P$.

This difference is only one of presentation.
Indeed, any such grid decomposition determines a bijection
\[
\tau:[d]^2\to P,
\qquad
\tau(i,j)=R_i\cap C_j,
\]
since each intersection $R_i\cap C_j$ consists of a unique point.
Via this bijection, the abstract grid decomposition may be identified with the
standard coordinate grid.
Conversely, any family of squares on $[d]^2$ yields such a grid decomposition by
taking its literal rows and columns.
Moreover, changing the identification $\tau$ amounts precisely to composing with
a permutation of the cell set $[d]^2$, and such a permutation preserves
balancedness, pairwise orthogonality, and the anti-Latin property.
Thus passing between the square-side and geometry-side descriptions involves no
loss of information.

\begin{theorem}\label{TH4}
Let $\Pi$ be an affine plane of order $d$.
Then the following are equivalent.
\begin{enumerate}
\item[(i)] There exist $d+1$ mutually orthogonal anti-Latin squares of order $d$
on the symbol set $[d]$ and a bijection
$\tau:[d]^2\to P$
such that $\tau$ carries the associated $(d,d+1)$-net of these squares onto the
affine-plane net of $\Pi$.
\item[(ii)] The plane $\Pi$ admits an anti-coordinate grid decomposition.
\end{enumerate}
\end{theorem}

\begin{proof}
Assume first \textup{(i)}.
Let $A^{(1)},\dots,A^{(d+1)}$ be the given mutually orthogonal anti-Latin
squares on $[d]^2$, and let $\tau:[d]^2\to P$ be a bijection carrying their
associated $(d,d+1)$-net onto the affine-plane net of $\Pi$.
For each $i\in[d]$, set $R_i:=\tau(\{i\}\times[d])$, and for each $j\in[d]$,
set $C_j:=\tau([d]\times\{j\})$.
Then $\mathcal R=\{R_i\}_{i\in[d]}$ and $\mathcal C=\{C_j\}_{j\in[d]}$ are
partitions of $P$ into $d$ sets of size $d$, and $|R_i\cap C_j|=1$ for all
$i,j\in[d]$, because $\tau$ is a bijection.
Thus $(\mathcal R,\mathcal C)$ is a grid decomposition of $P$.

It remains to show that this grid decomposition is anti-coordinate.
Fix a parallel class $\Gamma$ of $\Pi$.
Since $\tau$ identifies the associated net of the squares with the affine-plane
net of $\Pi$, the lines of $\Gamma$ are exactly the images under $\tau$ of the
symbol classes of one of the squares, say $A^{(t)}$.
Now each row of $A^{(t)}$ contains a repeated symbol, because $A^{(t)}$ is
anti-Latin.
Hence the corresponding row-block $R_i$ cannot meet every line of $\Gamma$ in
exactly one point; in other words, $R_i$ is not a $\Gamma$-transversal.
The same argument applies to each column-block $C_j$.
Therefore $(\mathcal R,\mathcal C)$ is an anti-coordinate grid decomposition of
$\Pi$, proving \textup{(ii)}.

Conversely, assume \textup{(ii)}.
Let $\Gamma_1,\dots,\Gamma_{d+1}$ be the $d+1$ parallel classes of $\Pi$, and
write $\mathcal R=\{R_i\}_{i\in[d]}$ and $\mathcal C=\{C_j\}_{j\in[d]}$ for the
given anti-coordinate grid decomposition.
Since $|R_i\cap C_j|=1$ for all $i,j\in[d]$, there is a unique point
$x_{i,j}\in P$ with $x_{i,j}\in R_i\cap C_j$.
Thus we obtain a bijection $\tau:[d]^2\to P$ by $\tau(i,j)=x_{i,j}$.

For each $t\in\{1,\dots,d+1\}$, choose a bijection $\lambda_t:\Gamma_t\to[d]$,
and define a $d\times d$ matrix $A^{(t)}$ on the symbol set $[d]$ by letting
$A^{(t)}_{i,j}$ be the label under $\lambda_t$ of the unique line of
$\Gamma_t$ containing $x_{i,j}$.

We claim that $A^{(1)},\dots,A^{(d+1)}$ are mutually orthogonal anti-Latin
squares.
First, each $A^{(t)}$ is balanced: if $a\in[d]$, then the entries equal to $a$
are exactly those $(i,j)$ for which $x_{i,j}$ lies on the line
$\lambda_t^{-1}(a)\in\Gamma_t$, and that line has exactly $d$ points.
Next, if $t\neq s$ and $a,b\in[d]$, then the lines
$\lambda_t^{-1}(a)\in\Gamma_t$ and $\lambda_s^{-1}(b)\in\Gamma_s$ meet in a
unique point of $P$, hence in a unique point $x_{i,j}$; therefore the ordered
pair $(a,b)$ appears exactly once among the superposed entries of $A^{(t)}$ and
$A^{(s)}$, so these two squares are orthogonal.
Finally, fix $t$.
Because $(\mathcal R,\mathcal C)$ is anti-coordinate, no $R_i$ is a
$\Gamma_t$-transversal.  Since $R_i$ has $d$ points and $\Gamma_t$ consists of
$d$ lines partitioning the point set, the pigeonhole principle implies that
some line of $\Gamma_t$ contains at least two points of $R_i$; equivalently, some symbol appears at least twice in row $i$ of
$A^{(t)}$.
The same argument applies to each column-block $C_j$, so every column of
$A^{(t)}$ also contains a repeated symbol.
Hence $A^{(t)}$ is anti-Latin.

Thus $A^{(1)},\dots,A^{(d+1)}$ are mutually orthogonal anti-Latin squares of
order $d$.
By construction, for each $t$ the symbol classes of $A^{(t)}$ are precisely the
preimages under $\tau$ of the lines in the parallel class $\Gamma_t$.
Therefore $\tau$ carries the associated $(d,d+1)$-net of these squares onto the
affine-plane net of $\Pi$, proving \textup{(i)}.
\end{proof}

The significance of Theorem~\ref{TH4} is that it separates two different layers
of the saturated problem.
On the one hand, the underlying incidence structure is not new: in the
saturated case, it is exactly an affine plane of order $d$.
On the other hand, the genuinely anti-Latin part is no longer hidden in the
entries of the squares, but is concentrated in the choice of a grid
decomposition, namely in the requirement that no row-block or column-block be
transversal to any parallel class.

Thus the existence of a saturated family is reduced to a two-step question:
first, whether an affine plane of order $d$ exists at all; and second, whether
that plane admits an anti-coordinate grid decomposition.
In this way, Theorem~\ref{TH4} connects the present problem with the classical
theory of affine planes and mutually orthogonal Latin squares, while at the
same time identifying the extra condition that is specific to anti-Latin
squares.
This reformulation is useful both conceptually and technically: it gives an
immediate structural obstruction in orders where affine planes do not exist,
and in the positive cases it provides the geometric framework in which later
coordinate and random-grid constructions are carried out.

\subsection{Directions induced by parallel classes on $[d]^2$}
\label{subsec:desargues}

We now rewrite the geometric condition of Theorem~\ref{TH4} on the fixed cell
set $[d]^2$.
Let $\Pi=(P,\mathcal L)$ be an affine plane of order $d$, and let
$\tau:[d]^2\to P$ be a bijection.
Via $\tau$, each parallel class of $\Pi$ may be viewed as a family of $d$
pairwise disjoint $d$-subsets of $[d]^2$.
For convenience, we refer to these transported parallel classes as
\emph{directions} on $[d]^2$.
This is only a change of language: no field structure on $[d]^2$ is assumed.
The terminology is motivated by the classical finite-geometry literature on
directions determined by point sets \cite{LovaszSchrijver,SzonyiDirections,KissSomlaiSpecialDirections}.

\begin{definition}\label{def:direction-complete}
Let $\Pi=(P,\mathcal L)$ be an affine plane of order $d$, let
$\tau:[d]^2\to P$ be a bijection, and let $X\subset[d]^2$ with $|X|=d$.
For a parallel class $\Gamma$ of $\Pi$, we say that $X$ \emph{determines}
$\Gamma$ if some line in $\Gamma$ contains at least two points of $\tau(X)$.
We say that $X$ is \emph{direction-complete} (with respect to $\tau$) if it
determines every parallel class of $\Pi$.
\end{definition}

\begin{remark}[Relation to the classical direction problem]
\label{rmk:classical-direction-problem}
For a point set $X$ in a Desarguesian affine plane, the classical direction
problem studies the set $D(X)$ of directions determined by differences of
distinct points of $X$, often with emphasis on lower bounds for $|D(X)|$ and
on the structure of sets attaining those bounds
\cite{LovaszSchrijver,SzonyiDirections,SomlaiRedei}.

When $|X|=q$, a direction is not determined precisely when $X$ meets each line
of the corresponding parallel class once.  Thus an undetermined direction is
exactly a direction in which $X$ is a transversal, or equivalently is
equidistributed among the parallel lines
\cite{KissSomlaiSpecialDirections}.  In this terminology, a
direction-complete block is a $q$-point set with no undetermined direction.

The anti-coordinate grid problem considered here imposes this extremal
condition simultaneously on all blocks of two orthogonal partitions.  This
simultaneous resolution requirement is not part of the usual single-set
direction problem.
\end{remark}
\begin{corollary}\label{CO2}
Let $\Pi=(P,\mathcal L)$ be an affine plane of order $d$.
Then the following are equivalent:
\begin{enumerate}
\item[(i)] There exist $d+1$ mutually orthogonal anti-Latin squares of order $d$
on the symbol set $[d]$ and a bijection $\tau:[d]^2\to P$ such that $\tau$
carries the associated $(d,d+1)$-net of these squares onto the affine-plane net
of $\Pi$.

\item[(ii)] There exist a bijection $\tau:[d]^2\to P$ and a grid decomposition
$(\mathcal R,\mathcal C)$ of $[d]^2$ such that every member of
$\mathcal R\cup\mathcal C$ is direction-complete with respect to $\tau$.
\end{enumerate}
Moreover, the passage between \textup{(i)} and \textup{(ii)} is constructive:
from a family as in \textup{(i)} one obtains $(\mathcal R,\mathcal C)$ by taking
the images under $\tau^{-1}$ of the row-blocks and column-blocks on $P$, and
conversely, from any $(\tau,\mathcal R,\mathcal C)$ as in \textup{(ii)} one
recovers a family of $d+1$ mutually orthogonal anti-Latin squares by labeling
the lines in each parallel class of $\Pi$.
\end{corollary}

\begin{proof}
Fix a bijection $\tau:[d]^2\to P$.
For a \(d\)-subset \(X\subset[d]^2\) and a parallel class \(\Gamma\),
\(X\) fails to determine \(\Gamma\) precisely when every line of
\(\Gamma\) contains at most one point of \(\tau(X)\). Since the \(d\)
lines of \(\Gamma\) partition \(P\) and \(|\tau(X)|=d\), this is
equivalent to saying that \(\tau(X)\) meets each line of \(\Gamma\)
in exactly one point.
Thus $X$ is direction-complete if and only if $\tau(X)$ is not transversal to
any parallel class of $\Pi$.
Applying this observation to all row-blocks and column-blocks, the statement is
just Theorem~\ref{TH4} transported from the point set $P$ to the fixed cell set
$[d]^2$ via $\tau$.
\end{proof}

In this form, the anti-coordinate condition is expressed entirely as a condition
on which parallel classes are determined by the row-blocks and column-blocks
when they are viewed on $[d]^2$.
This reformulation is useful because it turns the geometric non-transversality
condition into a concrete combinatorial test on $d$-subsets of $[d]^2$, while
still keeping the underlying affine-plane structure explicit.

\subsection{Finite-field model of directions}
\label{subsec:finite-field-directions}

We now specialize the preceding notion of directions to the standard affine
plane over a finite field.  Recall that, in
Subsection~\ref{subsec:desargues}, a direction on the fixed cell set $[d]^2$
means a parallel class of an affine plane after transport by a bijection from
$[d]^2$ to the point set of that plane.  When $d=q$ is a prime power, we may
take the affine plane to be the standard plane on $\mathbb F_q^2$.  In this
model, the parallel classes admit the usual slope parametrization.

For each finite slope $m\in\mathbb F_q$, let $\Gamma_m$ be the parallel class
consisting of the lines
\[
L_{m,b}:=\{(x,y)\in\mathbb F_q^2\mid y-mx=b\},
\qquad b\in\mathbb F_q.
\]
There is one further parallel class, namely the vertical class
$\Gamma_\infty$, consisting of the lines
$L_{\infty,a}:=\{(x,y)\in\mathbb F_q^2\mid x=a\}$, $a\in\mathbb F_q$.
Here $\infty$ is only a formal symbol used to denote the vertical direction;
it is not an element of $\mathbb F_q$.  Thus, in the finite-field model, the
parallel classes, and hence the directions in the sense of
Subsection~\ref{subsec:desargues}, are indexed by
$\mathbb F_q\cup\{\infty\}$.

Equivalently, introduce the maps
$\pi_m:\mathbb F_q^2\to\mathbb F_q$, $\pi_m(x,y):=y-mx$, for
$m\in\mathbb F_q$, and
$\pi_\infty:\mathbb F_q^2\to\mathbb F_q$, $\pi_\infty(x,y):=x$.
The fibres of $\pi_m$ are the lines in $\Gamma_m$, and the fibres of
$\pi_\infty$ are the lines in $\Gamma_\infty$.

Let $X\subset\mathbb F_q^2$ be a $q$-subset.  Then $X$ determines the finite
direction $\Gamma_m$ if and only if the multiset
$\{\pi_m(P)\mid P\in X\}$, equivalently
$\{y-mx\mid (x,y)\in X\}$, has a repeated value.  Similarly, $X$ determines
the vertical direction $\Gamma_\infty$ if and only if
$\{\pi_\infty(P)\mid P\in X\}$, equivalently
$\{x\mid (x,y)\in X\}$, has a repeated value.  Therefore, in this model,
$X$ is direction-complete if and only if this repeated-value condition holds
for every $u\in\mathbb F_q\cup\{\infty\}$.

Finally, suppose that the external cell set remains the fixed grid $[q]^2$.
After choosing a bijection $\tau:[q]^2\to\mathbb F_q^2$, a $q$-subset
$X\subset[q]^2$ determines the direction indexed by
$u\in\mathbb F_q\cup\{\infty\}$ if and only if the multiset
$\{\pi_u(\tau(i,j))\mid (i,j)\in X\}$ has a repeated value.  Thus
direction-completeness on $[q]^2$ is exactly the repeated-value test applied
to the image $\tau(X)$ in the finite-field affine plane.

\begin{proposition}[Finite-field form of the saturated anti-Latin condition]
\label{prop:finite-field-saturated-anti-latin}
Let $q$ be a prime power, and suppose that the affine plane associated with a
saturated family is identified with the standard affine plane $AG(2,q)$ on
$\mathbb F_q^2$.  Index the $q+1$ squares by
$u\in\mathbb F_q\cup\{\infty\}$ according to the corresponding parallel
classes.  Then a row-block or column-block $X$ satisfies the anti-Latin
repetition condition simultaneously in all $q+1$ squares if and only if, for
every $u\in\mathbb F_q\cup\{\infty\}$, the multiset
\[
 \{\pi_u(P):P\in X\}
\]
contains a repeated value.  Equivalently, $X$ determines all $q+1$
directions of $AG(2,q)$.
\end{proposition}
\begin{proof}
For a fixed direction $u$, the fibres of $\pi_u$ are exactly the lines of
that parallel class.  Two points of $X$ have the same $\pi_u$-value if and
only if they lie on the same line of direction $u$.  After the lines are
labeled by symbols, this is equivalent to a repeated symbol in the
corresponding row or column of the square indexed by $u$.  Applying this
observation to every $u$ proves the statement.  This is the finite-field form
of Corollary~\ref{CO2}.
\end{proof}

\section{Small-order cases and classifications}\label{sec:small-ex}
We now treat the remaining small orders directly.
The preceding sections develop the general structural framework for saturated
families, while Section~\ref{Sec3} gives its randomized
construction.
However, the cases $d=3$ and $d=4$ require separate arguments.
The different treatments of these two orders reflect a structural distinction: the extremal families for $d=3$ are not saturated, since $N_A(3)<d+1$, whereas those for $d=4$ are saturated and therefore fall within the affine-plane framework of Section~\ref{sec:geom}.
We first
introduce the equivalence notions used in the $d=3$ classification and then
discuss the two orders in turn.  The order-$4$ analysis is geometric rather
than a complete isomorphism classification of saturated families.

\subsection{Definitions of equivalence classes}\label{subsec:small-order-equivalence}

For the classification in the case $d=3$, we use the following notions of
equivalence.  They formalize the square-side operations that preserve the
anti-Latin property and mutual orthogonality.  The subsequent order-$4$
analysis is instead stated directly in the affine-plane language.

\begin{definition}[Weak isomorphism for families of mutually orthogonal anti-Latin squares]\label{def:weak-small-order}
Let $\mathcal{F}=(A^{(1)},\dots,A^{(K)})$ and
$\mathcal{F}'=(B^{(1)},\dots,B^{(K)})$ be two families of $K$
mutually orthogonal anti-Latin squares of order $d$.
We say that $\mathcal{F}$ and $\mathcal{F}'$ are \emph{weakly isomorphic}
if there exist a permutation $\sigma\in S_K$ and symbol permutations
$\pi_t\in S_d$ $(t=1,\dots,K)$ such that, for each $t$, if we define a square
$\widetilde{A}^{(t)}$ by
$
\widetilde{A}^{(t)}_{i,j}:=\pi_t\bigl(A^{(t)}_{i,j}\bigr)
$ for $i,j\in\{1,\dots,d\}$,
then
$\mathcal{F}'=(\widetilde{A}^{(\sigma(1))},\dots,\widetilde{A}^{(\sigma(K))})$.
In the weak notion, the $d\times d$ grid is fixed: rows and columns are not permuted.
\end{definition}

\begin{definition}[Strong isomorphism for families of mutually orthogonal anti-Latin squares]\label{def:strong-small-order}
With the notation of Definition~\ref{def:weak-small-order}, we say that
$\mathcal F=(A^{(1)},\dots,A^{(K)})$ and
$\mathcal F'=(B^{(1)},\dots,B^{(K)})$ are \emph{strongly isomorphic}
if there exist a permutation $\sigma\in S_K$, symbol permutations
$\pi_t\in S_d$ $(t=1,\dots,K)$, and permutations
$\rho,\kappa\in S_d$ such that, for each $t$, if we define
$\widetilde{A}^{(t)}_{i,j}
:=
\pi_t\bigl(A^{(t)}_{\rho(i),\,\kappa(j)}\bigr)$,
then
$\mathcal F'
=
(\widetilde{A}^{(\sigma(1))},\dots,\widetilde{A}^{(\sigma(K))})$.
Thus, in addition to the weak operations, we allow a common row permutation
$\rho$ and a common column permutation $\kappa$, applied simultaneously to
all $K$ squares.

\end{definition}

\subsection{Classification in the case $d=3$}\label{subsec:d3}
\subsubsection{The exact value $N_A(3)=3$}\label{subsubsec:d3-exact}

We first determine the exact value of $N_A(3)$.  
By Theorem~\ref{thm:3}, we have
\(
N_A(3)\ge N_L(3)+1.
\)
Since $N_L(3)=2$, this gives the lower bound
\(
N_A(3)\ge 3.
\)
Thus it remains only to prove the upper bound $N_A(3)\le 3$.

We begin with a simple structural fact about anti-Latin squares of order $3$.

\begin{lemma}\label{lem:d3-no-constant-row}
Let $A$ be an anti-Latin square of order $3$.  
Then no row and no column is constant.  
Equivalently, each row and each column has the form
$(x,x,y)$
up to permutation of its three entries, with $x\neq y$.
\end{lemma}

\begin{proof}
Since $A$ is anti-Latin, every row and every column contains a repeated symbol.  
It remains only to exclude the possibility that all three entries in some row are equal.

Assume, for contradiction, that one row is $(x,x,x)$.  Since each symbol
appears exactly three times in the whole square, the symbol $x$ does not occur
in the other two rows.  Hence every entry in those two rows belongs to the
remaining two symbols, say $y$ and $z$.  In every column, the two entries
outside the constant row must be equal, because the entry in the constant row
is $x$ and the column must contain a repeated symbol.  Therefore the two
remaining rows are identical.  But then the total numbers of occurrences of
$y$ and $z$ in these two rows are both even, whereas each of $y$ and $z$ must
occur exactly three times in the whole square.
This is impossible.  
Hence no row is constant.  
The same argument applied to the transpose shows that no column is constant.
\end{proof}

The upper bound $N_A(3)\le 3$ is obtained by normalizing the first rows of a
pairwise orthogonal family.

\begin{lemma}\label{lem:d3-first-row-patterns}
Let $\mathcal F$ be a family of pairwise orthogonal anti-Latin squares of order $3$.  
After independent symbol relabelings, we may assume that the entry at $(0,0)$ is 
equal to $0$ in every square of $\mathcal F$.  
Then the first row of each square in $\mathcal F$ is one of
\[
(0,0,1),\quad (0,0,2),\quad (0,1,0),\quad (0,2,0),\quad (0,1,1),\quad (0,2,2),
\]
and these six patterns split into the three incompatible pairs
\[
\{(0,0,1),(0,0,2)\},\qquad
\{(0,1,0),(0,2,0)\},\qquad
\{(0,1,1),(0,2,2)\}.
\]
For each of these three pairs, at most one pattern can occur among the first rows
of the squares in $\mathcal F$.  
In particular,
the inequality
$|\mathcal F|\le 3$ holds.
Moreover, if $|\mathcal F|=3$, then after reordering the three squares and applying
independent symbol relabelings, their first rows may be normalized to
\[
(0,0,1),\qquad (0,1,0),\qquad (0,1,1).
\]
\end{lemma}

\begin{proof}
For each square $A\in\mathcal F$, choose a symbol permutation so that the entry
at $(0,0)$ becomes $0$.  
Since symbol relabeling preserves both the anti-Latin property and orthogonality,
we may assume that $A_{0,0}=0$ for every square $A\in\mathcal F$.

Because $A$ is anti-Latin, the first row contains a repeated symbol.  
By Lemma~\ref{lem:d3-no-constant-row}, the first row is not constant.  
Hence the first row must be one of
\[
(0,0,1),\quad (0,0,2),\quad (0,1,0),\quad (0,2,0),\quad (0,1,1),\quad (0,2,2).
\]

These six possibilities split into the three pairs
\[
\{(0,0,1),(0,0,2)\},\qquad
\{(0,1,0),(0,2,0)\},\qquad
\{(0,1,1),(0,2,2)\}.
\]
If two distinct squares had first rows belonging to the same pair, then one
ordered pair would be repeated within the first row of their superposition,
contradicting orthogonality.  
For example, if two squares had first rows $(0,0,1)$ and $(0,0,2)$, then the
ordered pair $(0,0)$ would appear twice in the first row of the superposition.
The other two pairs are excluded in exactly the same way.

Thus at most one pattern from each incompatible pair can occur.  
Since the six admissible first-row patterns are partitioned into three
incompatible pairs, it follows that
\(
|\mathcal F|\le 3.
\)

If $|\mathcal F|=3$, then exactly one pattern from each incompatible pair occurs.
After reordering the three squares and, if necessary, independently relabeling
the nonzero symbols in each square, these three patterns may be normalized to
\[
(0,0,1),\qquad (0,1,0),\qquad (0,1,1).
\]
This proves the lemma.
\end{proof}

\begin{corollary}\label{cor:d3-exact-value}
We have
\[
N_A(3)=3.
\]
\end{corollary}

\begin{proof}
The lower bound $N_A(3)\ge 3$ was established above.  
On the other hand, Lemma~\ref{lem:d3-first-row-patterns} shows that any pairwise
orthogonal family of anti-Latin squares of order $3$ has size at most $3$.  
Hence $N_A(3)=3$.
\end{proof}

\subsubsection{A normal form for anti-Latin squares of order $3$}
\label{subsubsec:d3-square-side-structure}

We now describe anti-Latin squares of order $3$ by a single normal form.

Throughout this subsection, all row and column indices are taken in
\(
\mathbb Z_3=\{0,1,2\},
\)
and all additions of indices are understood modulo $3$.

\begin{definition}[Symbol-support]\label{def:d3-support}
Let $A$ be a $3\times3$ array on the symbol set $\mathbb Z_3$.
For $x\in\mathbb Z_3$, we define the support of $x$ in $A$ by
\[
\operatorname{supp}_A(x)
:=
\{(i,j)\in\mathbb Z_3^2 \mid A_{i,j}=x\}.
\]
\end{definition}

\begin{definition}\label{def:d3-models-sign}
For
\(
(\varepsilon,\eta)\in\{\pm1\}^2\) and 
$\delta\in\mathbb Z_3$,
we define 
the $3\times3$ array $A^{(\varepsilon,\eta)}_\delta$ 
on $\mathbb Z_3$
by
\[
\operatorname{supp}_{A^{(\varepsilon,\eta)}_\delta}(x)
=
\Bigl\{
(x,\delta+\varepsilon\eta\,x),\,
(x+\varepsilon,\delta+\varepsilon\eta\,x),\,
(x,\delta+\varepsilon\eta\,x+\eta)
\Bigr\}
\qquad (x\in\mathbb Z_3).
\]
\end{definition}

\begin{definition}[Normalized anti-Latin squares of order $3$]
A $3\times 3$ anti-Latin square $A$ on the symbol set $\mathbb Z_3$ is said to
be normalized if, for each $i\in\mathbb Z_3$, the symbol $i$ appears exactly
twice in row $i$.
\end{definition}

Indeed, in any anti-Latin square of order $3$, the repeated symbols in the
three rows are pairwise distinct.  Consequently, there is a unique symbol
relabeling that sends the repeated symbol in row $i$ to $i$ for every
$i\in\mathbb Z_3$, and this relabeling makes the square normalized.

\begin{proposition}[Normal form for $d=3$]\label{prop:d3-normal-form}
Let $A$ be a normalized anti-Latin square of order $3$.
Then there exist unique 
\(
(\varepsilon,\eta)\in\{\pm1\}^2\)
and $\delta\in\mathbb Z_3$
such that
$A=A^{(\varepsilon,\eta)}_\delta$.
Consequently, the natural action of $S_3$ by symbol relabeling on
anti-Latin squares is free, every orbit contains exactly one normalized square,
and the number of anti-Latin squares of order $3$ is
$4\cdot3\cdot3!=72$.
\end{proposition}

\begin{proof}
By Lemma~\ref{lem:d3-no-constant-row}, each row and each column contains a
repeated symbol, but no row and no column is constant.

\smallskip
\noindent
\textbf{Step 1: the $L$-shape of each symbol-support.}
Since $A$ is normalized, the repeated symbol in row $i$ is $i$ for each
$i\in\mathbb Z_3$.  The repeated symbols in the three columns are pairwise
distinct: if two columns had the same repeated symbol, that symbol would occur
at least four times, contradicting balancedness.  Hence, for each
$i\in\mathbb Z_3$, there is a unique column whose repeated symbol is $i$;
we denote this column by $\lambda(i)$.
Then symbol $i$ appears twice in row $i$ and twice in column $\lambda(i)$.
Since symbol $i$ appears exactly three times in total, these two double
occurrences must overlap in one corner cell.  Hence there exist
$\varepsilon_i,\eta_i\in\{\pm1\}$ such that
\begin{equation}
\operatorname{supp}_A(i)
=
\{(i,\lambda(i)),\,
  (i+\varepsilon_i,\lambda(i)),\,
  (i,\lambda(i)+\eta_i)\}.
\label{eq:d3-L-shape-support}
\end{equation}

\smallskip
\noindent
\textbf{Step 2: the two signs are independent of \(i\).}
We show that the signs \(\varepsilon_i\) and \(\eta_i\) do not depend on \(i\).
Recall from Step 1 that
\[
\operatorname{supp}_A(i)
=
\{(i,\lambda(i)),\,
  (i+\varepsilon_i,\lambda(i)),\,
  (i,\lambda(i)+\eta_i)\}.
\]

First consider the row coordinates.  Since the three supports
\(\operatorname{supp}_A(i)\), \(i\in\mathbb Z_3\), partition the nine cells of
\(\mathbb Z_3^2\), the sum of the row coordinates over all cells is equal to the
sum of the row coordinates over these three supports.  Hence, in \(\mathbb Z_3\),
\[
0
=
\sum_{(r,c)\in\mathbb Z_3^2} r
=
\sum_{i\in\mathbb Z_3} \bigl(i+(i+\varepsilon_i)+i\bigr)
=
\sum_{i\in\mathbb Z_3} \varepsilon_i .
\]
Here we used \(3i=0\) in \(\mathbb Z_3\).  Since each
\(\varepsilon_i\) is either \(1\) or \(-1\), the equality
$\varepsilon_0+\varepsilon_1+\varepsilon_2=0
$ in $\mathbb Z_3$
is possible only when
$\varepsilon_0=\varepsilon_1=\varepsilon_2$.
We denote this common value by \(\varepsilon\in\{\pm1\}\).

The same argument applied to the column coordinates gives
\[
0
=
\sum_{(r,c)\in\mathbb Z_3^2} c
=
\sum_{i\in\mathbb Z_3}
\bigl(\lambda(i)+\lambda(i)+\lambda(i)+\eta_i\bigr)
=
\sum_{i\in\mathbb Z_3}\eta_i .
\]
Since each \(\eta_i\) is either \(1\) or \(-1\), it follows that
\(
\eta_0=\eta_1=\eta_2.
\)
We denote this common value by \(\eta\in\{\pm1\}\).

\smallskip
\noindent
\textbf{Step 3: the affine form of the corner map.}
After Step 2, the supports have the form
\[
 S_i:=\operatorname{supp}_A(i)
 =\{(i,\lambda(i)),(i+\varepsilon,\lambda(i)),
     (i,\lambda(i)+\eta)\}.
\]
Fix a row $i$.  The support $S_i$ occupies the two cells in columns
$\lambda(i)$ and $\lambda(i)+\eta$ of that row.  The only support whose
vertical leg can contribute a further cell to row $i$ is $S_{i-\varepsilon}$,
and that cell is
\[
 (i,\lambda(i-\varepsilon)).
\]
Because the three supports are disjoint and partition the nine cells, these
three column coordinates must be the three distinct elements of
$\mathbb Z_3$.  The unique element different from
$\lambda(i)$ and $\lambda(i)+\eta$ is $\lambda(i)-\eta$.  Hence
\[
 \lambda(i-\varepsilon)=\lambda(i)-\eta,
\]
or equivalently
\[
 \lambda(i+\varepsilon)-\lambda(i)=\eta.
\]
Since $\varepsilon^{-1}=\varepsilon$ in $\mathbb Z_3$, the function
$i\mapsto\lambda(i)-\varepsilon\eta i$ is constant.  Thus there is a unique
$\delta\in\mathbb Z_3$, namely $\delta=\lambda(0)$, such that
\[
 \lambda(i)=\delta+\varepsilon\eta i
 \qquad(i\in\mathbb Z_3).
\]
Substituting this expression into \eqref{eq:d3-L-shape-support} gives
\[
\operatorname{supp}_A(i)
=
\Bigl\{
(i,\delta+\varepsilon\eta i),
(i+\varepsilon,\delta+\varepsilon\eta i),
(i,\delta+\varepsilon\eta i+\eta)
\Bigr\}.
\]
Therefore $A=A^{(\varepsilon,\eta)}_\delta$.  The signs
$\varepsilon,\eta$ are determined by the orientations of the repeated row and
column occurrences, and $\delta=\lambda(0)$, so all three parameters are
unique.

Conversely, fix $(\varepsilon,\eta)\in\{\pm1\}^2$ and
$\delta\in\mathbb Z_3$, and define the three supports by the displayed
formula.  Put $\lambda(i)=\delta+\varepsilon\eta i$.  In row $i$, the support
of symbol $i$ occupies columns $\lambda(i)$ and $\lambda(i)+\eta$, while the
support of symbol $i-\varepsilon$ contributes the cell in column
\[
 \lambda(i-\varepsilon)=\lambda(i)-\eta.
\]
These are the three distinct columns of $\mathbb Z_3$.  Hence every cell in
every row belongs to exactly one support, so the three supports are pairwise
disjoint and cover $\mathbb Z_3^2$.  They therefore define a well-defined
balanced array.  Symbol $i$ occurs twice in row $i$ and twice in column
$\lambda(i)$, so every row and every column has a repeated symbol.  Thus the
array is a normalized anti-Latin square.

There are $4\cdot3$ normalized squares.  The natural action of $S_3$ by symbol
relabeling is free, and the unique-normalization observation preceding the
proposition shows that every orbit contains exactly one normalized square.
Therefore the total number of anti-Latin squares of order $3$ is
\[
 4\cdot3\cdot3!=72.
\]

\end{proof}
\subsubsection{Orthogonality and classification of families in the case $d=3$}
\label{subsubsec:d3-square-side-orthogonality}

We next determine exactly when two normalized model squares are orthogonal.

\begin{proposition}[Orthogonality criterion in normal form]
\label{prop:d3-orthogonality-criterion}
Let
\[
(\varepsilon,\eta),(\varepsilon',\eta')\in\{\pm1\}^2,
\qquad
\delta,\delta'\in\mathbb Z_3.
\]
Then $A^{(\varepsilon,\eta)}_\delta$ and
$A^{(\varepsilon',\eta')}_{\delta'}$ are orthogonal if and only if
$(\varepsilon,\eta)=(\varepsilon',\eta')$ and $\delta\neq\delta'$.
\end{proposition}

\begin{proof}
Write $S_i$ for the support of symbol $i$ in
$A^{(\varepsilon,\eta)}_\delta$, and $S'_j$ for the support of symbol $j$ in
$A^{(\varepsilon',\eta')}_{\delta'}$.  Orthogonality is equivalent to
\(
|S_i\cap S'_j|=1\) with 
$i,j\in\mathbb Z_3$
because $S_i\cap S'_j$ is exactly the set of cells carrying the ordered pair
$(i,j)$.

Assume first that $(\varepsilon,\eta)=(\varepsilon',\eta')$.  If
$\delta=\delta'$, then the two squares are identical, so
$|S_i\cap S'_i|=3$ for every $i$; hence they are not orthogonal.  

Now suppose that $\delta\neq\delta'$.  Since $\delta'-\delta\in\{\eta,-\eta\}$, we consider the two cases separately.
Fix $i\in\mathbb Z_3$.  The three cells occupied by symbol $i$ in
$A^{(\varepsilon,\eta)}_\delta$ are
\[
p_1=(i,\delta+\varepsilon\eta i),\qquad
p_2=(i+\varepsilon,\delta+\varepsilon\eta i),\qquad
p_3=(i,\delta+\varepsilon\eta i+\eta).
\]

First, let $\delta'=\delta+\eta$.
We determine the entries of $A^{(\varepsilon,\eta)}_{\delta'}$ at
$p_1,p_2,p_3$ by checking their symbol-supports.  Since
\[
p_1\in
\operatorname{supp}_{A^{(\varepsilon,\eta)}_{\delta'}}(i-\varepsilon),
\qquad
p_2\in
\operatorname{supp}_{A^{(\varepsilon,\eta)}_{\delta'}}(i+\varepsilon),
\qquad
p_3\in
\operatorname{supp}_{A^{(\varepsilon,\eta)}_{\delta'}}(i),
\]
the entries of the second square at $p_1,p_2,p_3$ are respectively
$i-\varepsilon$, $i+\varepsilon$, and $i$.

Indeed, for example,
\[
\operatorname{supp}_{A^{(\varepsilon,\eta)}_{\delta'}}(i-\varepsilon)
=
\{(i-\varepsilon,\delta+\varepsilon\eta i),\,
  (i,\delta+\varepsilon\eta i),\,
  (i-\varepsilon,\delta+\varepsilon\eta i+\eta)\},
\]
so $p_1=(i,\delta+\varepsilon\eta i)$ belongs to this support.  The other two
membership assertions are obtained in the same way from the support formula.

When $\delta'=\delta-\eta$, the same support check gives
\[
p_1\in
\operatorname{supp}_{A^{(\varepsilon,\eta)}_{\delta'}}(i),\qquad
p_2\in
\operatorname{supp}_{A^{(\varepsilon,\eta)}_{\delta'}}(i+\varepsilon),\qquad
p_3\in
\operatorname{supp}_{A^{(\varepsilon,\eta)}_{\delta'}}(i-\varepsilon).
\]
Thus the entries of the second square at $p_1,p_2,p_3$ are respectively
$i$, $i+\varepsilon$, and $i-\varepsilon$.

In either case, the three entries of the second square on $S_i$ are the three
distinct elements of $\mathbb Z_3$.  Therefore every symbol $j$ occurs exactly
once on $S_i$, so
\[
 |S_i\cap S'_j|=1
 \qquad(i,j\in\mathbb Z_3).
\]
Hence the two squares are orthogonal.

Conversely, suppose that
$(\varepsilon,\eta)\neq(\varepsilon',\eta')$.  
First assume that exactly one of the two signs differs.
Then
$\varepsilon\eta\neq\varepsilon'\eta'$, so the two sets
\[
C=\{(i,\delta+\varepsilon\eta i)\mid i\in\mathbb Z_3\},
\qquad
C'=\{(i,\delta'+\varepsilon'\eta' i)\mid i\in\mathbb Z_3\}
\]
meet in a unique point, say $(r,c)$.  
Then, by the choice of
$(r,c)$, we have
$c=\delta+\varepsilon\eta r=\delta'+\varepsilon'\eta' r$.
The support formula gives
\[
S_r=\{(r,c),(r+\varepsilon,c),(r,c+\eta)\},
\qquad
S'_r=\{(r,c),(r+\varepsilon',c),(r,c+\eta')\}.
\]
If $\varepsilon=\varepsilon'$ and $\eta\neq\eta'$, 
then $S_r$ and $S'_r$ both contain the two common points $(r,c)$ and
$(r+\varepsilon,c)$.
If $\eta=\eta'$ and
$\varepsilon\neq\varepsilon'$, then $S_r$ and $S'_r$ both contain
$(r,c)$ and $(r,c+\eta)$.  
In either case some intersection
$S_i\cap S'_j$ has at least two points, so the two squares are not orthogonal.

It remains to consider the case in which both signs differ.  Then
$(\varepsilon',\eta')=(-\varepsilon,-\eta)$.  Since
$\delta'-\delta$ is one of $0,\eta,-\eta$, the support formula gives,
respectively,
\[
\begin{array}{ll}
\delta'=\delta:
&
S_i\cap S'_{i+\varepsilon}
\supset
\{(i+\varepsilon,\delta+\varepsilon\eta i),
  (i,\delta+\varepsilon\eta i+\eta)\},\\[1mm]
\delta'=\delta+\eta:
&
S_i\cap S'_i
\supset
\{(i,\delta+\varepsilon\eta i),
  (i,\delta+\varepsilon\eta i+\eta)\},\\[1mm]
\delta'=\delta-\eta:
&
S_i\cap S'_{i+\varepsilon}
\supset
\{(i,\delta+\varepsilon\eta i),
  (i+\varepsilon,\delta+\varepsilon\eta i)\}.
\end{array}
\]
Thus in every subcase some intersection of symbol-supports contains at least
two points.  Orthogonality is impossible.  This proves the criterion.
\end{proof}

\begin{corollary}\label{cor:d3-twelve-orthogonal-squares}
For every anti-Latin square $A$ of order $3$, there exist exactly $12$
anti-Latin squares of order $3$ that are orthogonal to $A$.
\end{corollary}

\begin{proof}
By Proposition~\ref{prop:d3-normal-form}, after relabeling the symbols we may
write
\(
A=A^{(\varepsilon,\eta)}_\delta
\)
for unique
\(
(\varepsilon,\eta)\in\{\pm1\}^2\) and \(
\delta\in\mathbb Z_3\).
By Proposition~\ref{prop:d3-orthogonality-criterion}, the normalized arrays
orthogonal to $A$ are exactly
\(
A^{(\varepsilon,\eta)}_{\delta'}
\)
for $\delta'\in\mathbb Z_3\setminus\{\delta\}$,
so there are exactly two of them in normalized form.
Each has exactly $3!$ symbol relabelings, and symbol relabeling preserves
orthogonality.
Hence the total number of orthogonal mates of $A$ is
\(
2\cdot 3!=12\).
\end{proof}

\begin{corollary}\label{cor:d3-triples-by-sign}
Fix
\(
(\varepsilon,\eta)\in\{\pm1\}^2.
\)
Then the three squares
\(
A^{(\varepsilon,\eta)}_0,\ 
A^{(\varepsilon,\eta)}_1,\ 
A^{(\varepsilon,\eta)}_2
\)
form a mutually orthogonal triple of anti-Latin squares of order $3$.
Moreover, every mutually orthogonal triple of anti-Latin squares of order $3$
consists of three squares with a common sign pair $(\varepsilon,\eta)$, and
hence, after independent symbol relabelings and a permutation of the three
squares, it is equal to
\(
\{A^{(\varepsilon,\eta)}_0,\,
A^{(\varepsilon,\eta)}_1,\,
A^{(\varepsilon,\eta)}_2\}
\)
for a unique sign pair $(\varepsilon,\eta)$.
\end{corollary}

\begin{proof}
The first statement follows immediately from
Proposition~\ref{prop:d3-orthogonality-criterion}.
Conversely, let
\(
\mathcal F=\{A,B,C\}
\)
be a mutually orthogonal triple.
By Proposition~\ref{prop:d3-normal-form}, each of $A,B,C$ is, after symbol
relabeling, of the form $A^{(\varepsilon,\eta)}_\delta$.
By Proposition~\ref{prop:d3-orthogonality-criterion}, orthogonality is possible
only between squares with the same sign pair and with distinct parameters.
Hence all three squares have a common sign pair $(\varepsilon,\eta)$, and since
\(
\mathbb Z_3=\{0,1,2\},
\)
their parameters must be exactly $0,1,2$ in some order.
This proves the claim.
\end{proof}

\begin{theorem}[Classification for $d=3$]\label{thm:d3-classification}
For mutually orthogonal triples of anti-Latin squares of order $3$,
there are exactly four weak-isomorphism classes under
Definition~\ref{def:weak-small-order} and exactly one strong-isomorphism class
under Definition~\ref{def:strong-small-order}, without requiring simultaneous transposition.
More precisely, representatives of the four weak-isomorphism classes may be
taken to be
\[
\{A^{(1,1)}_0,A^{(1,1)}_1,A^{(1,1)}_2\},\qquad
\{A^{(-1,1)}_0,A^{(-1,1)}_1,A^{(-1,1)}_2\},
\]
\[
\{A^{(1,-1)}_0,A^{(1,-1)}_1,A^{(1,-1)}_2\},\qquad
\{A^{(-1,-1)}_0,A^{(-1,-1)}_1,A^{(-1,-1)}_2\}.
\]
\end{theorem}

\begin{proof}
By Corollary~\ref{cor:d3-exact-value}, every maximal pairwise orthogonal family
has size $3$.
By Corollary~\ref{cor:d3-triples-by-sign}, every such triple is weakly
isomorphic to
\(
\{A^{(\varepsilon,\eta)}_0,A^{(\varepsilon,\eta)}_1,A^{(\varepsilon,\eta)}_2\}
\)
for a unique sign pair $(\varepsilon,\eta)$.
Since weak isomorphism fixes the external grid, the sign pair is preserved.
Hence there are exactly four weak-isomorphism classes.

For strong isomorphism, we may additionally apply common row and column
permutations.
A common row reversal changes $\varepsilon$ to $-\varepsilon$, while a common
column reversal changes $\eta$ to $-\eta$.
Hence any sign pair can be transformed into any other sign pair.
Therefore the four weak-isomorphism classes merge into a single
strong-isomorphism class.
\end{proof}

\subsection{The case $d=4$: extremal value and the geometry of saturated anti-Latin families}
\label{subsec:d4}
We now treat the case $d=4$.  We first determine the extremal value
$N_A(4)$ by an explicit saturated family.  Since the extremal value is
$d+1=5$, Section~\ref{sec:geom} then identifies the five symbol partitions of
an extremal family with the five parallel classes of an affine plane of order
$4$.  Our purpose here is to go beyond that general correspondence and
describe what the anti-Latin condition forces specifically in order $4$.

More precisely, we analyze the four-point sets corresponding to the rows and
columns of the original arrays, determine the symbol-multiplicity patterns
that they induce simultaneously in the five squares, and describe how the
exceptional double repetitions can be distributed among the four rows and,
separately, among the four columns.  The resulting discussion gives a complete affine classification of one
row-block or column-block and restricts the possible row and column
double-repetition profiles.  We distinguish throughout between the general
saturated correspondence inherited from Section~\ref{sec:geom} and the
additional conclusions that rely specifically on $d=4$.  We do not give an
isomorphism classification of all saturated families.

\subsubsection{The exact value $N_A(4)=5$}
\label{subsubsec:d4-exact}

We next record an explicit saturated family.  Its verification is finite and
direct: each displayed array is balanced, every row and every column contains
a repeated symbol, and the superposition of any two arrays contains each
ordered pair exactly once.

\begin{example}[A saturated family of order $4$]
\label{ex:d4-orbit10}
The following five arrays form a mutually orthogonal family of anti-Latin
squares of order $4$:
\begin{align*}
A^{(1)}&=\left(
\begin{array}{cccc}
0&0&1&1\\
0&0&1&1\\
2&2&3&3\\
2&3&2&3
\end{array}
\right),
&A^{(2)}&=\left(
\begin{array}{cccc}
0&1&0&2\\
2&3&1&3\\
0&2&0&1\\
3&2&1&3
\end{array}
\right),
&A^{(3)}&=\left(
\begin{array}{cccc}
0&1&1&2\\
3&2&0&3\\
3&0&2&3\\
1&1&2&0
\end{array}
\right),\\
A^{(4)}&=\left(
\begin{array}{cccc}
0&1&2&0\\
2&3&3&1\\
3&1&1&0\\
0&3&2&2
\end{array}
\right),
&A^{(5)}&=\left(
\begin{array}{cccc}
0&1&2&1\\
3&2&3&0\\
1&2&3&2\\
3&0&0&1
\end{array}
\right).
\end{align*}
\end{example}

Example~\ref{ex:d4-orbit10} gives $N_A(4)\ge5$.  Since $N_L(4)=3$,
Theorem~\ref{thm:NA-upper-via-MOLS} gives
\[
 N_A(4)\le N_L(4)+2=5.
\]

\begin{corollary}
\label{cor:d4-exact}
We have $N_A(4)=5$.
\end{corollary}
Since $5=d+1$, every extremal family in order $4$ is saturated.
Section~\ref{sec:geom} shows that the five symbol partitions of such a family
form the five parallel classes of an affine plane of order $4$.  The
order-$4$ question addressed below is therefore more specific than the
general geometric reformulation: we determine what the anti-Latin condition
forces on each of the four rows and four columns when they are viewed as
four-point subsets of the associated affine plane.

\subsubsection{Specialization of the general finite-field model}
\label{subsubsec:d4-geometric-recap}
We now specialize the finite-field model of
Subsection~\ref{subsec:finite-field-directions}.  Fix an identification of
the affine plane associated with the saturated family with the standard
affine plane $AG(2,4)$ on $\mathbb F_4^2$, where
\[
 \mathbb F_4=\{0,1,\omega,\omega+1\},\qquad \omega^2+\omega+1=0.
\]
Since the characteristic is two,
\[
 \omega^2=\omega+1,\qquad \omega(\omega+1)=1,
\]
so $\omega^{-1}=\omega+1$ and $(\omega+1)^{-1}=\omega$.
The five parallel classes are indexed by
\[
 0,\quad1,\quad\omega,\quad\omega+1,\quad\infty.
\]
For $u\in\mathbb F_4$, the four lines of direction $u$ are
\[
 L_{u,b}:=\{(x,y):y-ux=b\},\qquad b\in\mathbb F_4,
\]
and the four vertical lines are
\[
 L_{\infty,a}:=\{(x,y):x=a\},\qquad a\in\mathbb F_4.
\]

Under the identification of the common $4\times4$ cell grid with
$\mathbb F_4^2$, each row becomes a four-point subset, called a row-block,
and each column becomes a four-point subset, called a column-block.  By
Proposition~\ref{prop:finite-field-saturated-anti-latin}, which is the
finite-field form of the general saturated correspondence developed in
Section~\ref{sec:geom}, every row-block and every column-block determines all
five directions of $AG(2,4)$.

This conclusion is inherited from the general theory rather than being
specific to order $4$.  The additional order-$4$ structure begins with the
numerical fact that a row-block or column-block contains four points and
hence six unordered pairs, whereas $AG(2,4)$ has five directions.  As shown
below, this forces one direction to occur twice and leads, on the anti-Latin
side, to a unique square with symbol-multiplicity pattern $(2,2)$.

For distinct points $p=(x_1,y_1)$ and $q=(x_2,y_2)$, define
\[
\operatorname{dir}(p,q):=
\begin{cases}
\infty,&x_1=x_2,\\[1mm]
\dfrac{y_2-y_1}{x_2-x_1},&x_1\ne x_2.
\end{cases}
\]
Thus, in this specialization, a four-point set is direction-complete
precisely when its six unordered pairs of distinct points collectively
determine all five directions of $AG(2,4)$.

\subsubsection{Additional local structure in order $4$}
\label{subsubsec:d4-blocks}
Let $X\subset\mathbb F_4^2$ be a direction-complete set of four points.  Its
six unordered pairs of distinct points determine directions in
$\mathbb F_4\cup\{\infty\}$.  We say that three points are \emph{collinear}
if they lie on a common affine line.

\begin{lemma}[Direction multiplicities of a row or column block]
\label{lem:d4-direction-multiplicities}
Let $X$ be a direction-complete four-point subset of $AG(2,4)$.  Among
the six unordered pairs of distinct points of $X$, one direction is
determined by exactly two pairs, and each of the other four directions is
determined by exactly one pair.

Moreover, no affine line contains three points of $X$.  The two unordered
pairs of points that determine the direction occurring twice have no point
in common.
\end{lemma}
\begin{proof}
A four-point set has exactly six unordered pairs.  Since $X$ determines all
five directions, every direction occurs at least once among these six pairs.
Hence one direction occurs twice and each of the remaining four directions
occurs once.  Conversely, this multiplicity pattern clearly determines all
five directions.  If an affine line contained three points of $X$, the three
pairs formed by those points would have the same direction, which is
impossible.  For the same reason, the two pairs determining the direction
that occurs twice cannot have a common point.
\end{proof}
\begin{remark}[Arc-theoretic interpretation]
A direction-complete four-point block in $AG(2,4)$ has no three collinear
points and is therefore a $4$-arc in the standard terminology of finite
geometry \cite{Dembowski,Hirschfeld}.  In the projective completion
$PG(2,4)$, the five affine directions are represented by the five points of
the line at infinity.  The six secants joining pairs of points of the block
meet the line at infinity in all five of these points.  Exactly one point at
infinity is reached by two secants; this point represents the doubled
direction.

Thus Proposition~\ref{prop:d4-single-affine-orbit} is not a classification of
arbitrary $4$-arcs.  It classifies those affine $4$-arcs whose secants cover
the entire line at infinity.
\end{remark}

The unique direction determined by two of the six unordered pairs is called
the \emph{doubled direction} of $X$ and is denoted by $\delta(X)$.  This term
should not be confused with the notion of a special direction used in the
finite-geometric literature \cite{KissSomlaiSpecialDirections}.  Here all
five directions are determined; the doubled direction records the unique
direction represented by two of the six unordered pairs.  Among the
four parallel lines of direction $\delta(X)$, exactly two meet $X$, and each
of these two lines contains exactly two points of $X$.  For every direction
$u\ne\delta(X)$, one line of direction $u$ contains two points of $X$, two
further lines contain one point each, and the remaining line contains no
point of $X$.

For example, consider
\[
 X_*:=\{(0,0),(1,0),(0,1),(\omega,1)\}.
\]
Its six pair directions are
\[
\begin{array}{c|c}
\text{pair}&\text{direction}\\
\hline
(0,0),(1,0)&0\\
(0,0),(0,1)&\infty\\
(0,0),(\omega,1)&\omega+1\\
(1,0),(0,1)&1\\
(1,0),(\omega,1)&\omega\\
(0,1),(\omega,1)&0.
\end{array}
\]
Thus $X_*$ determines all five directions, and its doubled direction is $0$.

\begin{proposition}[Symbol multiplicities in a row or column]
\label{prop:d4-row-column-multiplicity}
Let
\[
 \mathcal F=\{A^{(u)}:u\in\mathbb F_4\cup\{\infty\}\}
\]
be a saturated family of five mutually orthogonal anti-Latin squares of
order $4$, where the squares are indexed by the five directions of the
associated affine plane.  Let $X$ be the row-block corresponding to one row,
or the column-block corresponding to one column.  Then there is a unique
direction $\delta(X)$ such that the four entries of
$A^{(\delta(X))}$ in that row or column have multiplicity pattern $(2,2)$.
For every $u\ne\delta(X)$, the corresponding four entries of $A^{(u)}$ have
multiplicity pattern $(2,1,1)$.
\end{proposition}
\begin{proof}
For direction $\delta(X)$, exactly two lines meet $X$, with two points on
each line.  After these lines are labeled by symbols, two distinct symbols
therefore occur twice each, giving pattern $(2,2)$.  For every other
direction, the line-intersection sizes are $2,1,1,0$, so one symbol occurs
twice and two further symbols occur once each, giving pattern $(2,1,1)$.
\end{proof}
Here pattern $(2,2)$ means that two distinct symbols occur twice each,
whereas pattern $(2,1,1)$ means that one symbol occurs twice and two further
symbols occur once each.  The existence of a repeated symbol in every square
is inherited from the general saturated correspondence in
Section~\ref{sec:geom}.  The uniqueness statement is specific to order $4$:
for each row and each column, exactly one of the five squares has pattern
$(2,2)$, while each of the other four squares has pattern $(2,1,1)$.  This
additional order-$4$ property permits the following anti-Latin invariants.
\begin{definition}[Row and column double-repetition profiles]
\label{def:d4-double-repetition-profiles}
Let
\[
 \mathcal F=\{A^{(u)}:u\in\mathbb F_4\cup\{\infty\}\}
\]
be a saturated family of five mutually orthogonal anti-Latin squares of
order $4$, indexed by the five directions of the associated affine plane.
For each row $i$, let $u_{\rm row}(i)$ be the unique index $u$ for which row
$i$ of $A^{(u)}$ has symbol-multiplicity pattern $(2,2)$.  The multiplicity
pattern of
\[
 u_{\rm row}(0),\quad u_{\rm row}(1),\quad
 u_{\rm row}(2),\quad u_{\rm row}(3)
\]
is called the \emph{row double-repetition profile} of $\mathcal F$.

For each column $j$, define $u_{\rm col}(j)$ analogously as the unique index
$u$ for which column $j$ of $A^{(u)}$ has pattern $(2,2)$.  The multiplicity
pattern of
\[
 u_{\rm col}(0),\quad u_{\rm col}(1),\quad
 u_{\rm col}(2),\quad u_{\rm col}(3)
\]
is called the \emph{column double-repetition profile} of $\mathcal F$.
\end{definition}
By Proposition~\ref{prop:d4-row-column-multiplicity}, $u_{\rm row}(i)$ is
exactly the doubled direction $\delta(R_i)$ of the corresponding row-block,
and $u_{\rm col}(j)$ is exactly the doubled direction $\delta(C_j)$ of the
corresponding column-block.  Hence the two anti-Latin profiles can be computed
geometrically from
\[
 \delta(R_0),\ldots,\delta(R_3)
 \quad\text{and}\quad
 \delta(C_0),\ldots,\delta(C_3),
\]
respectively.  These profiles are coarse invariants of the row and column
repetition structure under permutations of the five squares, independent
symbol relabelings, and common row or column permutations.  They do not
determine the saturated family or its compatible pair up to isomorphism.

\begin{proposition}[Affine normal form of a row or column block]
\label{prop:d4-single-affine-orbit}
Every direction-complete four-point subset of $\mathbb F_4^2$ can be
carried by an affine change of coordinates onto
\[
 X_*:=\{(0,0),(1,0),(0,1),(\omega,1)\}.
\]
\end{proposition}

\begin{proof}
Let $X$ be direction-complete.  By
Lemma~\ref{lem:d4-direction-multiplicities}, exactly two unordered pairs of
points determine $\delta(X)$, and these pairs have no point in common.
Translate one endpoint of the first pair to $(0,0)$.  The common direction
determined by these two pairs gives a nonzero difference vector, and the
displacement between the two parallel lines containing them is linearly
independent of that vector.
Choose an invertible linear map sending these two vectors to $(1,0)$ and
$(0,1)$, respectively.  Under the resulting affine transformation, the first
pair becomes $\{(0,0),(1,0)\}$, while the second pair becomes
$\{(0,1),(t,1)\}$ for some $t\in\mathbb F_4$.
Thus $X$ is affinely equivalent to
\[
 X_t:=\{(0,0),(1,0),(0,1),(t,1)\}.
\]
In this normalization, $t=0$ makes the last two displayed points coincide, and
$t=1$ gives a parallelogram, for which two directions occur twice and hence
only four directions are determined.  Therefore
$t\in\{\omega,\omega+1\}$.  For either of these two values, the four
non-horizontal pairs determine the four remaining directions; equivalently,
using $\omega^2=\omega+1$, their slopes together with the vertical direction
are all distinct.  Finally, the affine map
\[
(x,y)\longmapsto (\omega x,y+1)
\]
sends $X_{\omega+1}$ onto $X_\omega=X_*$. Hence both remaining possibilities are affinely equivalent to $X_*$, and therefore there is a single affine orbit.
\end{proof}

Proposition~\ref{prop:d4-single-affine-orbit} completely classifies the
individual blocks occurring in an anti-coordinate grid decomposition.  The
next question is how four such blocks can partition the sixteen-point plane.

\subsubsection{Geometric restrictions on the anti-Latin profiles}
\label{subsubsec:d4-partitions}
For a saturated anti-Latin family, the four row-blocks form a partition
\[
 \mathcal R=\{R_0,R_1,R_2,R_3\}
\]
of the sixteen points, and the four column-blocks form another partition
\[
 \mathcal C=\{C_0,C_1,C_2,C_3\}.
\]
Definition~\ref{def:d4-double-repetition-profiles} assigns a row and a column
double-repetition profile directly to the anti-Latin family.  We now use the
associated geometry to restrict these profiles.

For this purpose, call a partition
\[
 \mathcal P=\{P_0,P_1,P_2,P_3\}
\]
of $\mathbb F_4^2$ \emph{direction-complete} if every block $P_i$ is a
direction-complete four-point set.  For such a partition, we consider the
multiplicity pattern of its four doubled directions
\[
 \delta(P_0),\quad\delta(P_1),\quad
 \delta(P_2),\quad\delta(P_3).
\]
The five formal possibilities are
\[
\begin{array}{c|c}
\text{doubled directions of the four blocks}&\text{multiplicity pattern}\\
\hline
u,u,u,u &(4)\\
u,u,u,v &(3,1)\\
u,u,v,v &(2,2)\\
u,u,v,w &(2,1,1)\\
u,v,w,z &(1,1,1,1),
\end{array}
\]
where different letters denote different directions.  When $\mathcal P$ is
$\mathcal R$ or $\mathcal C$, this geometric multiplicity pattern is exactly
the row or column double-repetition profile of the anti-Latin family.

Classical direction results concern the direction set of one point set.  At
the present stage, however, the direction set of each block is already fixed
to be the full set of five directions.  The relevant new information is how
the unique excess secant direction varies across the four blocks of a
partition.  This gives a partition-level invariant that is not visible from
the number of directions determined by any individual block.

Fix a direction $u$, and write the four parallel lines of direction $u$ as
\[
 \ell_0^u,\quad\ell_1^u,\quad\ell_2^u,\quad\ell_3^u.
\]
For a partition $\mathcal P=\{P_0,P_1,P_2,P_3\}$, define
\[
 M^u(\mathcal P)
 :=\bigl(|P_i\cap\ell_j^u|\bigr)_{0\le i,j\le3}.
\]
The $i$th row of this matrix records how the four points of the block $P_i$
are distributed among the four lines of direction $u$.  Each row sum is $4$
because $P_i$ has four points.  Each column sum is also $4$ because every line
$\ell_j^u$ has four points and the four blocks partition the plane.

If $u=\delta(P_i)$, then two lines of direction $u$ contain two points of
$P_i$ each, while the other two lines contain no point of $P_i$.  Thus the
$i$th matrix row is a rearrangement of
\[
 (2,2,0,0).
\]
If $u\ne\delta(P_i)$, then one line contains two points of $P_i$, two lines
contain one point each, and one line contains no point.  The corresponding
matrix row is therefore a rearrangement of
\[
 (2,1,1,0).
\]

\begin{proposition}[Multiplicity patterns of doubled directions]
\label{prop:d4-partition-patterns}
For every direction-complete partition, the multiplicity pattern of its four
doubled directions is one of
\[
 (4),\qquad(2,2),\qquad(2,1,1),\qquad(1,1,1,1).
\]
The pattern $(3,1)$ cannot occur.
\end{proposition}
\begin{proof}
Assume that three blocks have one common doubled direction, say $u$, while
the fourth block has a different doubled direction.  In the matrix
$M^u(\mathcal P)$, the three rows corresponding to the first three blocks are
rearrangements of $(2,2,0,0)$.  Their entrywise sum is therefore even in every
column.  The remaining row is a rearrangement of $(2,1,1,0)$ and has exactly
two odd entries.  After this row is added, exactly two column sums are odd.
This is impossible because every column sum of $M^u(\mathcal P)$ equals $4$.
Hence the pattern $(3,1)$ does not occur, leaving the four patterns displayed
in the statement.
\end{proof}

\begin{corollary}[Possible row and column double-repetition profiles]
\label{cor:d4-anti-latin-profiles}
For every saturated family of five mutually orthogonal anti-Latin squares of
order $4$, both its row double-repetition profile and its column
double-repetition profile belong to
\[
 (4),\qquad(2,2),\qquad(2,1,1),\qquad(1,1,1,1).
\]
In particular, one square cannot have pattern $(2,2)$ in exactly three of the
four rows while another square has pattern $(2,2)$ in the remaining row.  The
analogous statement holds for the four columns.
\end{corollary}
\begin{proof}
By Proposition~\ref{prop:d4-row-column-multiplicity}, the distinguished square
for a row or column is indexed by the doubled direction of the corresponding
block.  The conclusion therefore follows from
Proposition~\ref{prop:d4-partition-patterns}.
\end{proof}

\subsubsection{Geometric realizations of the four multiplicity patterns}
\label{subsubsec:d4-profile-realizations}
It remains to show that each of the four multiplicity patterns in
Proposition~\ref{prop:d4-partition-patterns} occurs for a direction-complete
partition.  These examples are geometric candidates for a row-block or
column-block partition of a saturated anti-Latin family.  They do not, by
themselves, assert that the displayed partition has a compatible mate.  We
verify the doubled directions directly.

For the verification tables below, if a block is displayed as
\[
 B=\{p_1,p_2,p_3,p_4\},
\]
we list its six pair directions in the order
\[
 \operatorname{dir}(p_1,p_2),\quad
 \operatorname{dir}(p_1,p_3),\quad
 \operatorname{dir}(p_1,p_4),\quad
 \operatorname{dir}(p_2,p_3),\quad
 \operatorname{dir}(p_2,p_4),\quad
 \operatorname{dir}(p_3,p_4).
\]
A row of a verification table contains all five directions and repeats exactly
the direction shown in its last column.  Thus the table verifies both that the
block is direction-complete and that the stated direction is its doubled
direction.

\begin{example}[A partition with profile $(4)$]
\label{ex:d4-pattern-4}
Define
\begin{align*}
P_0^{(4)}&:=
 \{(0,0),(0,1),(1,0),(1,\omega)\},\\
P_1^{(4)}&:=
 \{(0,\omega),(0,\omega+1),(1,1),(1,\omega+1)\},\\
P_2^{(4)}&:=
 \{(\omega,0),(\omega,1),(\omega+1,0),(\omega+1,\omega)\},\\
P_3^{(4)}&:=
 \{(\omega,\omega),(\omega,\omega+1),(\omega+1,1),
   (\omega+1,\omega+1)\}.
\end{align*}
The four sets are pairwise disjoint and their union is $\mathbb F_4^2$.
Therefore
\[
 \mathcal P^{(4)}
 :=\{P_0^{(4)},P_1^{(4)},P_2^{(4)},P_3^{(4)}\}
\]
is a partition.  Its pair directions are
\[
\begin{array}{c|c|c}
\text{block}&\text{directions of the six pairs}&\text{doubled direction}\\
\hline
P_0^{(4)}&\infty,0,\omega,1,\omega+1,\infty&\infty\\
P_1^{(4)}&\infty,\omega+1,1,\omega,0,\infty&\infty\\
P_2^{(4)}&\infty,0,\omega,1,\omega+1,\infty&\infty\\
P_3^{(4)}&\infty,\omega+1,1,\omega,0,\infty&\infty.
\end{array}
\]
Thus all four blocks have the same doubled direction, and
$\mathcal P^{(4)}$ has doubled-direction multiplicity pattern $(4)$.  Geometrically, this is
an admissible candidate for a row-block or column-block partition.  If it
occurs as a row-block partition of a saturated anti-Latin family, the same one of the five squares
has multiplicity pattern $(2,2)$ in all four rows, while each of the other
four squares has pattern $(2,1,1)$ in every row.  The same interpretation
holds for columns.
\end{example}

\begin{example}[A partition with profile $(2,2)$]
\label{ex:d4-pattern-22}
Define
\begin{align*}
P_0^{(2,2)}&:=
 \{(0,0),(0,1),(1,0),(1,\omega)\},\\
P_1^{(2,2)}&:=
 \{(0,\omega),(0,\omega+1),(1,1),(\omega,1)\},\\
P_2^{(2,2)}&:=
 \{(1,\omega+1),(\omega,0),(\omega+1,0),(\omega+1,\omega)\},\\
P_3^{(2,2)}&:=
 \{(\omega,\omega),(\omega,\omega+1),(\omega+1,1),
   (\omega+1,\omega+1)\}.
\end{align*}
Again the four sets are pairwise disjoint and cover $\mathbb F_4^2$, so they
form a partition $\mathcal P^{(2,2)}$.  The verification is
\[
\begin{array}{c|c|c}
\text{block}&\text{directions of the six pairs}&\text{doubled direction}\\
\hline
P_0^{(2,2)}&\infty,0,\omega,1,\omega+1,\infty&\infty\\
P_1^{(2,2)}&\infty,\omega+1,\omega,\omega,1,0&\omega\\
P_2^{(2,2)}&1,\omega,\omega+1,0,\omega,\infty&\omega\\
P_3^{(2,2)}&\infty,\omega+1,1,\omega,0,\infty&\infty.
\end{array}
\]
The doubled directions are $\infty,\omega,\omega,\infty$, so its doubled-direction multiplicity pattern is $(2,2)$.  Geometrically,
this is an admissible candidate for a row-block or column-block partition.
If it occurs as a row-block partition of a saturated anti-Latin family, two of the five squares each have pattern
$(2,2)$ in two rows, while the remaining three squares have no row of pattern
$(2,2)$.  The same interpretation holds for columns.
\end{example}

\begin{example}[A partition with profile $(2,1,1)$]
\label{ex:d4-pattern-211}
Define
\begin{align*}
P_0^{(2,1,1)}&:=
 \{(0,0),(0,1),(1,0),(1,\omega)\},\\
P_1^{(2,1,1)}&:=
 \{(0,\omega),(0,\omega+1),(1,1),(\omega,1)\},\\
P_2^{(2,1,1)}&:=
 \{(1,\omega+1),(\omega,0),(\omega+1,0),
   (\omega+1,\omega+1)\},\\
P_3^{(2,1,1)}&:=
 \{(\omega,\omega),(\omega,\omega+1),(\omega+1,1),
   (\omega+1,\omega)\}.
\end{align*}
These four sets form a partition $\mathcal P^{(2,1,1)}$.  Their pair directions
are
\[
\begin{array}{c|c|c}
\text{block}&\text{directions of the six pairs}&\text{doubled direction}\\
\hline
P_0^{(2,1,1)}&\infty,0,\omega,1,\omega+1,\infty&\infty\\
P_1^{(2,1,1)}&\infty,\omega+1,\omega,\omega,1,0&\omega\\
P_2^{(2,1,1)}&1,\omega,0,0,\omega+1,\infty&0\\
P_3^{(2,1,1)}&\infty,\omega+1,0,\omega,1,\infty&\infty.
\end{array}
\]
The doubled directions are $\infty,\omega,0,\infty$.  Hence one direction
occurs twice and two further directions occur once, giving doubled-direction multiplicity pattern $(2,1,1)$.  Geometrically,
this is an admissible candidate for a row-block or column-block partition.
If it occurs as a row-block partition of a saturated anti-Latin family, one square has pattern $(2,2)$ in two
rows, and two further squares have pattern $(2,2)$ in one row each.  The same
interpretation holds for columns.
\end{example}

\begin{example}[A partition with profile $(1,1,1,1)$]
\label{ex:d4-pattern-1111}
Define
\begin{align*}
P_0^{(1,1,1,1)}&:=
 \{(0,0),(0,1),(1,0),(1,\omega)\},\\
P_1^{(1,1,1,1)}&:=
 \{(0,\omega),(0,\omega+1),(\omega,0),(\omega+1,0)\},\\
P_2^{(1,1,1,1)}&:=
 \{(1,1),(\omega,1),(\omega,\omega+1),
   (\omega+1,\omega+1)\},\\
P_3^{(1,1,1,1)}&:=
 \{(1,\omega+1),(\omega,\omega),(\omega+1,1),
   (\omega+1,\omega)\}.
\end{align*}
These four sets form a partition $\mathcal P^{(1,1,1,1)}$.  The verification is
\[
\begin{array}{c|c|c}
\text{block}&\text{directions of the six pairs}&\text{doubled direction}\\
\hline
P_0^{(1,1,1,1)}&\infty,0,\omega,1,\omega+1,\infty&\infty\\
P_1^{(1,1,1,1)}&\infty,1,\omega+1,\omega,1,0&1\\
P_2^{(1,1,1,1)}&0,\omega+1,1,\infty,\omega,0&0\\
P_3^{(1,1,1,1)}&\omega,1,\omega+1,\omega+1,0,\infty&\omega+1.
\end{array}
\]
The four doubled directions are mutually distinct.  Therefore its doubled-direction multiplicity pattern is $(1,1,1,1)$.
Geometrically, this is an admissible candidate for a row-block or column-block
partition.  If it occurs as a row-block partition of a saturated anti-Latin
family, four distinct squares each have pattern
$(2,2)$ in one row, while the fifth square has pattern $(2,1,1)$ in all four
rows.  The same interpretation holds for columns.
\end{example}

Combining Proposition~\ref{prop:d4-partition-patterns} with
Examples~\ref{ex:d4-pattern-4}--\ref{ex:d4-pattern-1111}, we obtain the exact
list of possible profiles.

\begin{corollary}[Realized doubled-direction multiplicity patterns]
\label{cor:d4-partition-patterns}
A direction-complete partition has exactly one of the four
doubled-direction multiplicity patterns
\[
 (4),\qquad(2,2),\qquad(2,1,1),\qquad(1,1,1,1),
\]
and each of these four patterns occurs.
\end{corollary}
This result classifies only the multiplicities of the four doubled directions.
It does not assert that two partitions with the same multiplicity pattern are
isomorphic.  Indeed, the relative positions of their blocks may contain
further geometric information not recorded by this pattern.  Nor do the
examples alone show that every pattern occurs as the row or column profile of
an actual saturated anti-Latin family, because compatibility with a second
partition is an additional condition.
\subsubsection{Compatible row and column structures}
\label{subsubsec:d4-compatible-pairs}

The preceding analysis concerns one partition of the affine plane into four
direction-complete blocks.  The compatibility between the row-block and
column-block partitions is not specific to order $4$: it is part of the
general anti-coordinate grid structure established in
Theorem~\ref{TH4}.  In the present $AG(2,4)$ representation, each row of the
common cell grid meets each column in exactly one cell, so every row-block
meets every column-block in exactly one point.

\begin{definition}[Compatible pair]
\label{def:d4-compatible-pair}
Let
\[
 \mathcal R=\{R_0,R_1,R_2,R_3\},\qquad
 \mathcal C=\{C_0,C_1,C_2,C_3\}
\]
be direction-complete partitions of $\mathbb F_4^2$.  We call $\mathcal R$
the \emph{row-block partition} and $\mathcal C$ the \emph{column-block
partition}.  The ordered pair $(\mathcal R,\mathcal C)$ is called a
\emph{compatible pair} if
\[
 |R_i\cap C_j|=1
 \qquad(i,j\in[4]).
\]
\end{definition}

For a compatible pair, let $x_{i,j}$ denote the unique point of
$R_i\cap C_j$.  Since the four row-blocks and four column-blocks are
partitions of the same sixteen-point set, the map
\[
 [4]^2\longrightarrow\mathbb F_4^2,
 \qquad
 (i,j)\longmapsto x_{i,j},
\]
is a bijection.  Thus the points $x_{i,j}$ may be regarded as the cells of an
external $4\times4$ grid, with $R_i$ as its $i$th row and $C_j$ as its $j$th
column.

\begin{proposition}[Compatible-pair characterization]
\label{prop:d4-compatible-pair-characterization}
Giving a saturated family of five mutually orthogonal anti-Latin squares of
order $4$, together with an identification of its associated affine plane
with $AG(2,4)$, is equivalent to giving the following data:
\begin{enumerate}
\item a compatible pair $(\mathcal R,\mathcal C)$ of direction-complete
partitions of $\mathbb F_4^2$;
\item for each of the five parallel classes, a labeling of its four lines by
the four symbols in $[4]$.
\end{enumerate}
\end{proposition}

\begin{proof}
Suppose first that a saturated family is given, together with the stated
identification of its associated affine plane with $AG(2,4)$.  For each
$i\in[4]$, let $R_i$ be the image in $AG(2,4)$ of row $i$ of the common
$4\times4$ cell grid, and for each $j\in[4]$, let $C_j$ be the image of column
$j$.  The four sets $R_i$ form a partition of the sixteen points, and so do
the four sets $C_j$.  Each row of the cell grid meets each column in exactly
one cell.  Therefore
\[
 |R_i\cap C_j|=1
 \qquad(i,j\in[4]).
\]
Moreover, fix one of the five directions. 
The corresponding anti-Latin square has at least one repeated symbol in every row and every column. 
Hence every row-block $R_i$ and every column-block $C_j$ contains two points on one line of that direction.

Since this holds for all five directions, all row-blocks
and column-blocks are direction-complete.  Therefore
$(\mathcal R,\mathcal C)$ is a compatible pair.  The symbols of each square
provide the required labels of the four lines in the corresponding parallel
class.
Conversely, suppose that a compatible pair and the five line labelings are
given.  Let $x_{i,j}$ be the unique point of $R_i\cap C_j$.  For each direction
$u\in\mathbb F_4\cup\{\infty\}$, define a $4\times4$ array $A^{(u)}$ by
\[
 A^{(u)}_{i,j}
 :=
 \text{the label of the unique line of direction $u$ containing $x_{i,j}$}.
\]
We verify the three required properties.

First, $A^{(u)}$ is balanced.  Each symbol labels one line of direction $u$,
and that line contains four points, so the symbol occurs exactly four times.
Second, if $u\ne v$, a line of direction $u$ and a line of direction $v$ meet
in exactly one point.  Hence every ordered pair of symbols occurs exactly
once in the superposition of $A^{(u)}$ and $A^{(v)}$, and the two arrays are
orthogonal.  Third, fix a row-block $R_i$ and a direction $u$.  Since $R_i$ is
direction-complete, it contains two points on one line of direction $u$.
The corresponding two entries in row $i$ of $A^{(u)}$ have the same line
label.  Thus every row contains a repeated symbol.  The same argument applied
to each column-block $C_j$ shows that every column contains a repeated symbol.
Therefore every $A^{(u)}$ is anti-Latin, and the five arrays form a saturated
family.

The two constructions are inverse up to the chosen line labels and the stated
identification with $AG(2,4)$.
\end{proof}
Thus every saturated family has an ordered pair consisting of its row and
column double-repetition profiles.  These profiles record genuine order-$4$
repetition structure, but they do not determine the compatible pair up to
isomorphism.

\subsubsection{Inherited general structure and additional order-$4$ consequences}
\label{subsubsec:d4-structural-summary}
We finally separate the structure inherited from the general saturated
correspondence from the additional conclusions obtained only in order $4$.

\begin{proposition}[General saturated structure specialized to order $4$]
\label{prop:d4-inherited-structure}
Let $\mathcal F$ be a saturated family of five mutually orthogonal anti-Latin
squares of order $4$.  Its five symbol partitions form the five parallel
classes of an affine plane of order $4$, and its row-block and column-block
partitions form a compatible anti-coordinate grid decomposition.  After
fixing an identification of the associated affine plane with $AG(2,4)$, each
row-block and each column-block determines all five directions.
\end{proposition}
\begin{proof}
The affine-plane and anti-coordinate-grid statements are
Theorem~\ref{TH4}.  The direction-completeness statement is the specialization
of Proposition~\ref{prop:finite-field-saturated-anti-latin} to $q=4$.
\end{proof}

\begin{theorem}[Additional structure specific to order $4$]
\label{thm:d4-saturated-structure}
Under the identification in
Proposition~\ref{prop:d4-inherited-structure}, the following statements hold.
\begin{enumerate}
\item Every row-block and column-block $X$ has a unique doubled direction
$\delta(X)$.
\item In the square indexed by $\delta(X)$, the four entries in the
corresponding row or column have multiplicity pattern $(2,2)$.  In each of
the other four squares, they have pattern $(2,1,1)$.
\item Every row-block and column-block is affinely equivalent to
\[
 \{(0,0),(1,0),(0,1),(\omega,1)\}.
\]
\item The row double-repetition profile and the column double-repetition
profile are each one of
\[
 (4),\qquad(2,2),\qquad(2,1,1),\qquad(1,1,1,1).
\]
In particular, profile $(3,1)$ is impossible.
\end{enumerate}
\end{theorem}
\begin{proof}
The first two statements follow from
Lemma~\ref{lem:d4-direction-multiplicities} and
Proposition~\ref{prop:d4-row-column-multiplicity}.  The third is
Proposition~\ref{prop:d4-single-affine-orbit}, and the fourth is
Corollary~\ref{cor:d4-anti-latin-profiles}.
\end{proof}

\begin{remark}[Scope of the order-$4$ analysis]
The compatible anti-coordinate grid and direction-completeness statements in
Proposition~\ref{prop:d4-inherited-structure} are inherited from the general
theory.  The genuinely order-$4$ conclusions are the unique doubled
direction, the $(2,2)$ versus $(2,1,1)$ distinction, the affine normal form,
and the restrictions on the row and column double-repetition profiles.  We
do not classify all direction-complete partitions or all compatible pairs up
to isomorphism.
\end{remark}

\subsection{Final synthesis}
\label{subsec:final-synthesis}

We now combine the preceding results and obtain the proof of Theorem \ref{thm:main-exact-formula}
for $N_A(d)$.

For $d=3$, Theorem~\ref{thm:d3-classification} gives $N_A(3)=3$. Since
$N_L(3)=2$, this yields
\(
N_A(3)=N_L(3)+1.
\)

For $d=4,5,7$, the explicit constructions of Example~\ref{ex:d4-orbit10}
and Example~\ref{EX57} show that
\(
N_A(d)\ge d+1.
\)
Since these orders are prime powers, we have $N_L(d)=d-1$, and therefore
Theorem~\ref{thm:NA-upper-via-MOLS} gives
\(
N_A(d)\le N_L(d)+2=d+1.
\)
Hence
\(
N_A(d)=N_L(d)+2\) for 
$d=4,5,7$.

For $d=6$ and for all $d\ge 8$, the conclusion follows directly from
Theorem~\ref{thm:NA-equals-NL-plus-2-for-all-d-ge-8}. This completes the proof
of Theorem \ref{thm:main-exact-formula}.

\section{Conclusion}\label{sec:conclusion}

We have determined the extremal size $N_A(d)$ of a family of mutually orthogonal
anti-Latin squares in terms of the classical Latin-square quantity $N_L(d)$.
At the general level, the paper proves the universal comparison
\[
N_L(d)+1\le N_A(d)\le N_L(d)+2
\qquad (d\ge 3),
\]
where the upper bound is obtained by passing through balanced matrices and the
lower bound by a deterministic permutation argument
(Theorems~\ref{thm:NA-upper-via-MOLS} and~\ref{thm:3}).
The main theorem then resolves this two-point window completely:
\[
N_A(3)=N_L(3)+1,
\qquad
N_A(d)=N_L(d)+2
\quad (d\ge 4).
\]
Thus, apart from the single exceptional order $d=3$, the anti-Latin extremal
problem differs from the classical Latin-square problem by a fixed shift of two.

For large orders, the argument is uniform.
The probabilistic construction shows that the upper bound is attained for all
$d\ge 8$, and the same method also covers the exceptional order $d=6$
(Theorem~\ref{thm:NA-equals-NL-plus-2-for-all-d-ge-8}).
Accordingly, the only genuinely order-specific work is confined to the remaining
small cases.

Beyond the exact formula for $N_A(d)$, the geometric reformulation
separates a classical component from an additional anti-Latin structure.  The
classical component is the passage from a complete net to an affine plane.
The additional structure identified by Theorem~\ref{TH4} is a pair of
orthogonal resolutions whose every block is non-transversal to every parallel
class, equivalently an anti-coordinate grid decomposition.  This is a
coordinate-free condition formulated purely in the incidence language of
points, lines, and parallel classes, and at the level of this structural
characterization no distinction between Desarguesian and non-Desarguesian
affine planes \cite{JohnsonDulmageMendelsohn1961} is used.

In this sense, the paper refines the classical correspondence among affine
planes, nets, orthogonal arrays, and mutually orthogonal Latin squares by
identifying an additional structure that is specific to the anti-Latin
setting
\cite{Bruck1951,Dembowski,Stinson,BoseBush,AbelColbournDinitzMOLS,JohnsonDulmageMendelsohn1961,Evans2018}.
After transport to the fixed cell set $[d]^2$, this extra structure becomes
the requirement that the relevant row-blocks and column-blocks be
direction-complete, thereby linking mutually orthogonal anti-Latin squares to
the finite-geometric theory of directions determined by point sets
\cite{LovaszSchrijver,SzonyiDirections,SomlaiRedei,KissSomlaiSpecialDirections}.
The resulting problem is not the classical problem of determining the
direction set of one point set, but a simultaneous decomposition problem for
two compatible families of point sets.

The low-order picture is then completed as follows.
For $d=3$, we prove $N_A(3)=3$ and classify orthogonal triples into four weak
classes and one strong class
(Theorems~\ref{thm:main-exact-formula} and~\ref{thm:d3-classification}).
For $d=4$, saturation occurs and the explicit family in
Example~\ref{ex:d4-orbit10} gives $N_A(4)=5$.  The affine-plane and compatible
anti-coordinate-grid structure is inherited from the general saturated
correspondence.  The additional order-$4$ analysis shows that every row and
column has a unique square with symbol-multiplicity pattern $(2,2)$, while the
other four squares have pattern $(2,1,1)$.  It also gives an affine normal
form for the corresponding four-point blocks and restricts the possible
distributions of the exceptional $(2,2)$ pattern among the four rows and,
separately, among the four columns.
For the remaining prime-power orders treated explicitly in the paper, namely
$d=5$ and $d=7$, the randomized-grid method produces concrete saturated
examples (Example~\ref{EX57}).
Together with the large-order theorem, this determines $N_A(d)$ for every
order $3\le d\le 9$.
For convenience, Table~\ref{tab:small-d-summary} summarizes the exact values of
$N_A(d)$ for $3\le d\le 9$ together with the main results and methods used in
each case.

\begin{table}[t]
\centering
\caption{Summary of the determined values of $N_A(d)$ for $3\le d\le 9$.}
\label{tab:small-d-summary}
\begin{tabular}{@{}ccp{5.2cm}p{5.2cm}@{}}
\toprule
$d$ &  $N_A(d)$ & result(s) used & method / route \\
\midrule
3 & $3$ & Theorem~\ref{thm:main-exact-formula} and Theorem~\ref{thm:d3-classification} & small-order classification; exclusion of a fourth orthogonal square \\
4 & $5$ & Example~\ref{ex:d4-orbit10} and Theorem~\ref{thm:NA-upper-via-MOLS} & explicit saturated family, general upper bound, and finite-geometric analysis \\
5 & $6$ & Example~\ref{EX57} and Theorem~\ref{thm:NA-upper-via-MOLS} & prime-power case: explicit saturated family, together with the upper bound $N_A(d)\le N_L(d)+2=d+1$ \\
6 & $3$ & Theorem~\ref{thm:NA-equals-NL-plus-2-for-all-d-ge-8} and the classical equality $N_L(6)=1$ & probabilistic attainment of the upper bound in the exceptional composite order $d=6$ \\
7 & $8$ & Example~\ref{EX57} and Theorem~\ref{thm:NA-upper-via-MOLS} & prime-power case: explicit saturated family, together with the upper bound $N_A(d)\le N_L(d)+2=d+1$ \\
8 & $9$ & Theorem~\ref{thm:NA-equals-NL-plus-2-for-all-d-ge-8} and the classical equality $N_L(8)=7$ & large-order theorem, specialized to the prime-power order $d=8$ \\
9 & $10$ & Theorem~\ref{thm:NA-equals-NL-plus-2-for-all-d-ge-8} and the classical equality $N_L(9)=8$ & large-order theorem, specialized to the prime-power order $d=9$ \\
\bottomrule
\end{tabular}
\end{table}

In summary, mutually orthogonal anti-Latin squares admit both a clean extremal description and a clean geometric description.  Numerically, apart from the single exceptional order $d=3$, the maximum number of mutually orthogonal anti-Latin squares is always exactly two more than the corresponding Latin-square extremal number.  Geometrically, the saturated case is characterized by affine planes equipped with anti-coordinate grid decompositions.  This leaves a natural direction for future work: to understand for which affine planes such decompositions exist and how this additional structure interacts with the classical theory of affine planes, nets, and mutually orthogonal Latin squares.

\bmhead{Acknowledgements}

A part of this work was inspired by responses to an earlier Japan Math-Contest problem
on anti-Latin squares. The authors thank 
Mr. Tomoki Anzai,
Mr. Takahiro Hori,
Mr. Kinari Miyazawa,
Mr. Nobuto Muraki,
Mr. Yuma Wakatsuki,
Mr. Kirato Yata, and
Mr. Shin Yumoto for their interest in the problem and for the initial
stimulus provided by their submitted solutions. The results, constructions,
and proofs presented in this paper were developed subsequently by the authors
in a substantially different form.

MH was supported in part by the Kayamori Foundation of Informational Science
Advancement.  Large language model tools were used as auxiliary aids in
preparing this manuscript, including assistance with exposition, literature
search, and exploratory discussions of possible approaches.  The manuscript
was written under the authors' direction, and the authors are responsible for
all mathematical content, proofs, references, and conclusions.

\begin{appendices}

\section{Explicit computational examples for $d=5$ and $d=7$}\label{app:explicit57}
This appendix presents explicit saturated families for $d=5$ and $d=7$
generated by the randomized-grid method of Section~\ref{sec:random-grid}.
A uniformly random bijection $\tau:[d]^2\to\mathbb F_d^2$ is sampled; the grid
decomposition $(\mathcal R,\mathcal C)$ is formed by pulling back coordinate
fibres; every block is required to determine all $d+1$ directions; then
Theorem~\ref{TH4} (ii)$\Rightarrow$(i) reconstructs the $d+1$ arrays by line
membership (see also Corollary~\ref{CO2}).  This realizes an anti-coordinate
grid and yields a saturated family.
\subsection{Provenance and reproducibility} 
The displayed families below were generated by the randomized-grid procedure
described in Section~\ref{sec:random-grid}. The supplementary script
\texttt{random\_grid\_generate\_d5\_d7.py} samples random bijections
\(
\tau:[d]^2\to \mathbb F_d^2
\)
and tests whether all standard row and column-blocks are direction-complete.
When such a bijection is found, the script constructs the \(d+1\) arrays by the
line-membership rule of Theorem~\ref{TH4}, namely by using the labels of the
lines in the \(d+1\) parallel classes of the affine plane over \(\mathbb F_d\).
With seed \(\mathtt{20260317}\), the recorded run found successful grid
decompositions at trial \(34\) for \(d=5\) and at trial \(3\) for \(d=7\).
The same script outputs the generated arrays and verifies that each square is
anti-Latin and that the resulting family is pairwise orthogonal.

\begin{example}\label{EX57}
For $d=5$ (six arrays, $K=6=d+1$):
{\scriptsize
\[\begin{aligned}
& 
\left(\begin{array}{ccccc}
2 & 1 & 0 & 3 & 1\\
2 & 2 & 1 & 4 & 3\\
4 & 2 & 4 & 4 & 0\\
4 & 1 & 0 & 1 & 3\\
0 & 3 & 2 & 3 & 0
\end{array}\right)
;\;
\left(\begin{array}{ccccc}
0 & 1 & 2 & 3 & 0\\
2 & 3 & 4 & 4 & 0\\
0 & 1 & 3 & 1 & 3\\
2 & 2 & 0 & 3 & 1\\
4 & 4 & 4 & 2 & 1
\end{array}\right)
;\;\left(\begin{array}{ccccc}
3 & 0 & 2 & 0 & 4\\
0 & 1 & 3 & 0 & 2\\
1 & 4 & 4 & 2 & 3\\
3 & 1 & 0 & 2 & 3\\
4 & 1 & 2 & 4 & 1
\end{array}\right)
;\;\left(\begin{array}{ccccc}
1 & 4 & 2 & 2 & 3\\
3 & 4 & 2 & 1 & 4\\
2 & 2 & 0 & 3 & 3\\
4 & 0 & 0 & 1 & 0\\
4 & 3 & 0 & 1 & 1
\end{array}\right)
;\;
\left(\begin{array}{ccccc}
4 & 3 & 2 & 4 & 2\\
1 & 2 & 1 & 2 & 1\\
3 & 0 & 1 & 4 & 3\\
0 & 4 & 0 & 0 & 2\\
4 & 0 & 3 & 3 & 1
\end{array}\right)
;\;
\left(\begin{array}{ccccc}
2 & 2 & 2 & 1 & 1\\
4 & 0 & 0 & 3 & 3\\
4 & 3 & 2 & 0 & 3\\
1 & 3 & 0 & 4 & 4\\
4 & 2 & 1 & 0 & 1
\end{array}\right)
\end{aligned}\]
}
For $d=7$ (eight arrays, $K=8=d+1$):
{\scriptsize

\[\begin{aligned}
&\left(\begin{array}{ccccccc}
2 & 6 & 4 & 2 & 2 & 5 & 6\\
2 & 5 & 0 & 5 & 1 & 3 & 6\\
5 & 4 & 0 & 0 & 3 & 2 & 5\\
4 & 1 & 0 & 4 & 6 & 3 & 6\\
3 & 1 & 2 & 0 & 2 & 4 & 3\\
0 & 3 & 3 & 1 & 1 & 6 & 1\\
4 & 0 & 1 & 5 & 4 & 5 & 6
\end{array}\right)
;\;
\left(\begin{array}{ccccccc}
4 & 5 & 2 & 1 & 2 & 2 & 0\\
6 & 6 & 5 & 1 & 6 & 2 & 1\\
0 & 0 & 2 & 4 & 4 & 3 & 4\\
1 & 4 & 1 & 3 & 6 & 0 & 4\\
5 & 3 & 5 & 3 & 0 & 6 & 1\\
6 & 6 & 3 & 5 & 2 & 3 & 0\\
4 & 0 & 1 & 5 & 5 & 3 & 2
\end{array}\right)
;\;
\left(\begin{array}{ccccccc}
2 & 6 & 5 & 6 & 0 & 4 & 1\\
4 & 1 & 5 & 3 & 5 & 6 & 2\\
2 & 3 & 2 & 4 & 1 & 1 & 6\\
4 & 3 & 1 & 6 & 0 & 4 & 5\\
2 & 2 & 3 & 3 & 5 & 2 & 5\\
6 & 3 & 0 & 4 & 1 & 4 & 6\\
0 & 0 & 0 & 0 & 1 & 5 & 3
\end{array}\right)
;\;
\left(\begin{array}{ccccccc}
0 & 0 & 1 & 4 & 5 & 6 & 2\\
2 & 3 & 5 & 5 & 4 & 3 & 3\\
4 & 6 & 2 & 4 & 5 & 6 & 1\\
0 & 2 & 1 & 2 & 1 & 1 & 6\\
6 & 1 & 1 & 3 & 3 & 5 & 2\\
6 & 0 & 4 & 3 & 0 & 5 & 5\\
3 & 0 & 6 & 2 & 4 & 0 & 4
\end{array}\right)\\
&
\left(\begin{array}{ccccccc}
5 & 1 & 4 & 2 & 3 & 1 & 3\\
0 & 5 & 5 & 0 & 3 & 0 & 4\\
6 & 2 & 2 & 4 & 2 & 4 & 3\\
3 & 1 & 1 & 5 & 2 & 5 & 0\\
3 & 0 & 6 & 3 & 1 & 1 & 6\\
6 & 4 & 1 & 2 & 6 & 6 & 4\\
6 & 0 & 5 & 4 & 0 & 2 & 5
\end{array}\right)
;\;
\left(\begin{array}{ccccccc}
3 & 2 & 0 & 0 & 1 & 3 & 4\\
5 & 0 & 5 & 2 & 2 & 4 & 5\\
1 & 5 & 2 & 4 & 6 & 2 & 5\\
6 & 0 & 1 & 1 & 3 & 2 & 1\\
0 & 6 & 4 & 3 & 6 & 4 & 3\\
6 & 1 & 5 & 1 & 5 & 0 & 3\\
2 & 0 & 4 & 6 & 3 & 4 & 6
\end{array}\right)
;\;
\left(\begin{array}{ccccccc}
1 & 3 & 3 & 5 & 6 & 5 & 5\\
3 & 2 & 5 & 4 & 1 & 1 & 6\\
3 & 1 & 2 & 4 & 3 & 0 & 0\\
2 & 6 & 1 & 4 & 4 & 6 & 2\\
4 & 5 & 2 & 3 & 4 & 0 & 0\\
6 & 5 & 2 & 0 & 4 & 1 & 2\\
5 & 0 & 3 & 1 & 6 & 6 & 0
\end{array}\right)
;\;
\left(\begin{array}{ccccccc}
6 & 4 & 6 & 3 & 4 & 0 & 6\\
1 & 4 & 5 & 6 & 0 & 5 & 0\\
5 & 4 & 2 & 4 & 0 & 5 & 2\\
5 & 5 & 1 & 0 & 5 & 3 & 3\\
1 & 4 & 0 & 3 & 2 & 3 & 4\\
6 & 2 & 6 & 6 & 3 & 2 & 1\\
1 & 0 & 2 & 3 & 2 & 1 & 1
\end{array}\right).
\end{aligned}\]
}
\end{example}

\begin{remark}[Randomized grid provenance and verification]\label{rmk:rg-verify}
Each family above is the direct output of the randomized-grid reconstruction with the seeds indicated in the provenance paragraph. A minimal script that regenerates such a family (and the accompanying solver--free verifier) is provided with the supplementary material; the verifier checks the anti-Latin row/column property and pairwise orthogonality exactly as encoded in the proof of Theorem~\ref{TH4}. 
\end{remark}

\end{appendices}

\section*{Ethics approval and consent to participate}
Not applicable.

\section*{Consent for publication}
Not applicable.

\section*{Availability of supporting data}
For the randomized-grid constructions in the cases $d=5$ and $d=7$, the
supplementary material contains
\texttt{random\_grid\_generate\_d5\_d7.py} and
\texttt{random\_grid\_d5\_d7\_output.json}.
The script \texttt{random\_grid\_generate\_d5\_d7.py} implements the
randomized-grid procedure described in Section~\ref{sec:random-grid}: it
samples random bijections
\(
\tau:[d]^2\to \mathbb F_d^2,
\)
tests direction-completeness of all standard row and column-blocks, and
constructs the \(d+1\) arrays by the line-membership rule of
Theorem~\ref{TH4}. The JSON output records the successful bijections, trial
numbers, generated families, and verification results. 

\section*{Competing interests}
There are no competing interests.

\end{document}